\documentclass[11pt]{article}
\usepackage{amssymb,amsmath,amsthm,amsfonts,verbatim}
\usepackage[all,cmtip]{xy}
\input diagxy
\usepackage{hyperref}
\usepackage[margin=1.25in]{geometry}
\usepackage{todonotes}
\usepackage{qtree}
\usepackage{attrib}

\newtheorem{theorem}{Theorem}[section]
\newtheorem{proposition}[theorem]{Proposition}
\newtheorem{corollary}[theorem]{Corollary}
\newtheorem{lemma}[theorem]{Lemma}

\newtheorem{convention}[theorem]{Convention}

\theoremstyle{definition}

\newtheorem{definition}[theorem]{Definition}
\newtheorem{example}[theorem]{Example}

\newtheorem{notation}[theorem]{Notation}

\newtheorem{claim}{Claim}
\newtheorem{problem}[theorem]{Problem}
\newtheorem{remark}[theorem]{Remark}

\numberwithin{equation}{section}

\newcommand{\into}{\hookrightarrow}

\renewcommand{\S}{\underline{S}}

\newcommand{\Ic}{\mathcal{I}}
\newcommand{\Lc}{\mathcal{L}}
\newcommand{\Mc}{\mathcal{M}}

\newcommand{\Oc}{\mathcal{O}}

\newcommand{\Hc}{\mathcal{H}}

\newcommand{\xdownarrow}[1]{%
  {\left\downarrow\vbox to #1{}\right.\kern-\nulldelimiterspace}
}

\newcommand{\Ab}{\mathbb{A}}

\newcommand{\C}{\mathbb{C}}

\newcommand{\Fb}{\mathbb{F}}
\newcommand{\F}{\mathbb{F}}
\newcommand{\Gb}{\mathbb{G}}

\newcommand{\Nb}{\mathbb{N}}

\newcommand{\Pb}{\mathbb{P}}

\newcommand{\Rb}{\mathbb{R}}

\newcommand{\Zb}{\mathbb{Z}}
\newcommand{\Z}{\mathbb{Z}}

\newcommand{\Fpbar}{\overline{\Fb}_p}

\renewcommand{\a}{\mathbf{a}}
\renewcommand{\b}{\mathbf{b}}
\newcommand{\x}{\mathbf{x}}
\newcommand{\tb}{\tilde{\b}}

\newcommand{\Aut}{Aut}

\DeclareMathOperator{\Spec}{Spec}

\newcommand{\Gr}{Gr}

\newcommand{\Tc}{\mathcal{T}}

\DeclareMathOperator{\RD}{RD}

\DeclareMathOperator{\PGL}{PGL}

\DeclareMathOperator{\Stab}{Stab}

\DeclareMathOperator{\ev}{ev}

\DeclareMathOperator{\Gal}{Gal}

\DeclareMathOperator{\Mon}{Mon}

\DeclareMathOperator{\End}{End}

\DeclareMathOperator{\pr}{pr}

\DeclareMathOperator{\Image}{Image}
\DeclareMathOperator{\FW}{F}

\renewcommand{\O}{{\rm O}}

\title{Tschirnhaus transformations after Hilbert}
\author{Jesse Wolfson\thanks{Supported in part by NSF Grant DMS-1811846.}}

\begin{document}
\maketitle
\begin{abstract}
	In this paper, we use enumerative geometry to simplify the formula for the roots of the general one-variable polynomial of degree $n$, for all $n$. More precisely, let $\RD(n)$ denote the minimum $d$ for which there exists a formula for the roots of the general degree $n$ polynomial using only algebraic functions of $d$ or fewer variables. In 1927, Hilbert sketched how the 27 lines on a cubic surface could be used to construct a 4-variable formula for the general degree 9 polynomial (implying $\RD(9)\le 4$). In this paper, we turn Hilbert's sketch into a general method. We show this method produces best-to-date upper bounds on $\RD(n)$ for all $n$, improving earlier results of Hamilton, Sylvester, Segre and Brauer.
\end{abstract}

\section{Introduction}
The goal of this paper is to use enumerative geometry to produce simplest-to-date formulas for the roots of the general one-variable polynomial of degree $n$, for all $n$. Consider the problem of finding the roots of a polynomial
\begin{equation*}
    z^n+a_1z^{n-1}+\cdots+a_n=0
\end{equation*}
in terms of the coefficients $a_1,\ldots,a_n$. A priori, the assignment
\begin{equation*}
    (a_1,\ldots,a_n)\mapsto\{z~|~z^n+a_1z^{n-1}+\cdots+a_n=0\}
\end{equation*}
is an algebraic function of $n$ (complex) variables, and it is natural to ask whether there exists a formula using only algebraic functions of $d$ or fewer variables. Call the minimum such $d$ the {\em resolvent degree} and denote this by $\RD(n)$ (see Section~\ref{s:rd} for a precise definition, and \cite{FW} for a detailed treatment). At present, no nontrivial lower bounds for $\RD(n)$ are known. The best general upper bounds in the literature are due to Brauer \cite{Br}, who uses methods dating to Tschirnhaus \cite{Ts} to prove that $\RD(n)\le n-r$ for $n\ge (r-1)!+1$. As Brauer remarks, his bounds are not optimal for small $r$.\footnote{Brauer's first improvement over prior bounds occurs for $r=7$.}

In this paper we take a different approach to bounding $\RD(n)$, inspired by a geometric argument of Hilbert.  In \cite{Hi}, Hilbert sketches how the 27 lines on a cubic surface can be used to produce a 4-variable formula for the general degree 9 polynomial, i.e. $\RD(9)\le 4$.  We turn Hilbert's sketch into a general method, whereby lines on cubic surfaces are replaced by $r$-planes on degree $d$ hypersurfaces in $\Pb^m$ for appropriate choices of $r$, $d$ and $m$. This defines an explicit increasing function $\FW\colon \Nb\to\Nb$ (Definition~\ref{d:FW}) for which we prove the following:
\begin{theorem}\mbox{}
    Let $\FW\colon\Nb\to\Nb$ be the function defined in Definition~\ref{d:FW}.
    \begin{enumerate}
        \item For all $r$ and all $n\ge\FW(r)$, $\RD(n)\le n-r$.
        \item For all $r$, $n=\FW(r)$ is the least value for which we know $\RD(n)\le n-r$ to hold.\footnote{i.e. $n=\FW(r)$ is the least value for which $\RD(n)\le n-r$ is currently proven to hold in any of the literature of which we are aware. Note that G. Chebotarev \cite{Ch} claimed to have extended an argument of Wiman  \cite{Wi} to conclude $\RD(n)\le n-6$ for $n\ge 21$. His proof has gaps similar to those observed by Dixmier \cite{Di} in the arguments of Hilbert and Wiman, namely he takes for granted that certain forms are generic, when they are not.} In particular, the initial values are given by
            \begin{equation*}
                \begin{array}{cccccccc}
                r & 1 & 2 & 3 & 4 & 5 & 6 & 7\\
                \hline
                \FW(r) & 2 & 3 & 4 & 5 & 9 & 41 & 121
                \end{array}
            \end{equation*}
        \item Writing $B(r)=(r-1)!+1$ for Brauer's bound, then
                \begin{equation*}
                    \lim_{r\rightarrow\infty} B(r)/\FW(r)=\infty.
                \end{equation*}
    \end{enumerate}
\end{theorem}
\noindent
The first statement appears as Theorem~\ref{thm:beatbrauer} below, while the last two appear as Theorem~\ref{t:compare}.
\begin{remark}\mbox{}
    \begin{enumerate}
        \item The construction of $\FW$, the proof that $\FW(5)=9$ and that this implies $\RD(9)\le 4$ marks the first rigorous construction of the 4-variable formula for the general degree 9 sketched by Hilbert in \cite{Hi}.\footnote{Rigorous 4-variable formulas have been previously constructed by Segre \cite{Seg1} and Dixmier \cite{Di}.}
        \item The first improvement over prior bounds occurs at $\FW(6)=41$. Previously, Sylvester proved  \cite[p. 485]{Sy} that for $n\ge 44$, $\RD(n)\le n-6$.
    \end{enumerate}
\end{remark}
Besides the general interest in obtaining simpler formulas for polynomials, we hope this paper spurs work on two questions. For the first, we quote Dixmier \cite[p. 90]{Di}\footnote{n.b. Dixmier writes ``$s(n)$'' for our $\RD(n)$.}:
\begin{quote}
    ``Every reduction of $\RD(n)$ would be serious progress. In particular, it is time to know if $\RD(6)=1$ or $\RD(6)=2$.''
    \attrib{Dixmier, 1993}
\end{quote}
While the present methods cannot touch Hilbert's Sextic Conjecture ($\RD(6)=2$), they do contribute to Dixmier's call to lower the possible values of $\RD(n)$. They also contribute to a problem first posed (as far as we are aware) by Segre \cite[III.5]{Seg2}:
\begin{problem}
    Understand the large $n$ behavior of $\RD(n)$.
\end{problem}
\noindent
As a clearer understanding of Segre's problem comes into view, we look forward to seeing the present bounds lowered in turn.

\paragraph*{Remarks on the Proof.}
Given a polynomial
\[
    p(z)=z^n+a_1z^{n-1}+\cdots+a_n=\prod_{i=1}^n(z-z_i),
\]
a {\em Tschirnhaus transformation} is a ``change of variables''
\[
    y=\sum_{j=0}^{n-1} b_jz^j.
\]
This gives a new polynomial
\[
    q(y)=\prod_i(y-\sum_{j=0}^{n-1} b_jz_i^j)=y^n+c_1y^{n-1}+\cdots+c_n,
\]
and we can ask for Tschirnhaus transformations which normalize the resulting polynomial so that, e.g.
\begin{equation}\label{e:Tintro}
    c_1=\cdots=c_k=0.
\end{equation}
The space of all $(b_0,\ldots,b_{n-1})$ such that the conditions \eqref{e:Tintro} are satisfied forms an affine cone, and the projectivization gives a complete intersection
\[
    T^n_{1\cdots k}\subset\Pb^{n-1};
\]
when the superscript $n$ is clear from context, we suppress it and write $T_{1\cdots k}$. If we can find a point of $T_{1\cdots k}$ over a convenient extension of $\C(a_1,\ldots,a_n)$, e.g. one defined using only algebraic functions of at most $d$ variables, then we can write a formula for the general degree $n$ polynomial using only functions of at most $d$ variables and the algebraic function
\[
    (c_{k+1},\ldots,c_n)\mapsto \{y~|~y^n+c_{k+1}y^{n-k-1}+\cdots+c_n=0\},
\]
This, together with a final rational change of coordinates, gives an upper bound
\[
    \RD(n)\le\max\{d,n-k-1\}.
\]
In \cite{Hi}, Hilbert sketched how to use the 27 lines on a smooth cubic surface to find points on $T_{1234}$ for $n=9$: Here, $T_1\subset\Pb^8$ is a hyperplane, and thus $T_{12}$ is a quadric 6-fold in $T_1\cong \Pb^7$. Over a solvable extension $L/\C(a_1,\ldots,a_9)$, every smooth quadric contains a 3-plane $P$ in $\Pb^7$. The intersection of this 3-plane $P$ with $T_{123}$ is a cubic surface, and this gives a map from $\Spec(L)$ to the moduli of cubic surfaces. Since every smooth cubic surface has 27 lines, and the moduli space of cubic surfaces is 4-dimensional, the algebraic function which assigns a line to a cubic surface is a function of at most 4-variables. Given a line on our cubic surface $P\cap T_{123}$, we can then intersect it with $T_{1234}$ to get a quartic polynomial in one variable, and by adjoining radicals, we can find a point on $T_{1234}(L')$, where $L'/\C(a_1,\ldots,a_9)$ is defined using algebraic functions of at most $d=4$ variables.

As Dixmier observed \cite[\S 8]{Di}, the argument above is incomplete. In particular, Hilbert takes for granted that the family of cubic surfaces $P\cap T_{123}$ is sufficiently generic. Letting $\Hc_{3,3}$ denote the parameter space of cubic surfaces and $\Mc_{3,3}$ the (coarse) moduli space of smooth cubic surfaces, Hilbert essentially assumes that the above map
\[
    \Spec(L)\to \Hc_{3,3}
\]
lands in the locus where the rational map
\[
    \Hc_{3,3}\dashrightarrow \Mc_{3,3}
\]
is well-defined.\footnote{n.b. Hilbert actually assumes that the generic member of the family $P\cap T_{123}$ admits a ``pentahedral form'', but one can weaken this as above without any loss in the argument.} The principal geometric contribution of this paper is to show that for all $n$, the family of ``Tschirnhaus hypersurfaces'' needed for Hilbert's argument (and its generalization to arbitrary degrees) is generically smooth; see Theorem~\ref{t:compint}.

Beyond this, we need two fundamental post-Hilbert advances to convert Hilbert's sketch into a general method.  The first is Merkurjev and Suslin's theorem on Severi-Brauer varieties \cite[Theorem 16.1]{MS}, which allows us to trivialize the Severi-Brauer varieties which arise in Hilbert's argument by adjoining radicals.\footnote{Neither Hilbert nor Dixmier comment on this gap in Hilbert's argument.} The second is a theorem of Hochster--Laksov \cite{HL} which allowed Waldron \cite[Theorem 1.6]{Wa} (see also \cite[Theorem 1.2]{St}) to show that {\em every} degree $d$ hypersurface in $\Pb^N$ contains an $r$-plane when an appropriate dimension count is non-negative. Given these, we can generalize Hilbert's sketch to  explicitly construct the function $\FW$ and obtain the bounds on $\RD(n)$ stated above.

\paragraph*{Outline of the Paper.}
In Section~\ref{s:tschirnci} we introduce the Tschirnhaus complete intersections and study their geometry. In Section~\ref{s:tschirnt}, we recall the geometric perspective on Tschirnhaus transformations, and connect this to the Tschirnhaus complete intersections. In Section~\ref{s:rd}, we develop the necessary results about the resolvent degree of a dominant map needed to implement Hilbert's idea for general degrees $n$. This extends the treatment of resolvent degree of generically finite dominant maps in \cite{FW}. In Section~\ref{s:beatbrauer}, we prove the upper bounds for $\RD(n)$ and compare them to Brauer's. In Appendix~\ref{a:bounds}, we give explicit values for the function $\FW(r)$ discussed above. In Appendix~\ref{a:hist}, we review the history of the search for simple formulas for the general degree $n$ polynomial  and the summarize the major prior work to date.

\paragraph*{Conventions}
Throughout the paper, by a {\em variety} over a field $K$ or over $\Z$, we mean a reduced, separated, not-necessarily irreducible $K$ or $\Z$-scheme. For maps of varieties $X\to Z$ and $Y\to Z$, we will use the notation $Y|_X$ to denote the fiber product $X\times_Z Y$.

\paragraph*{Acknowledgements.} First, I thank Benson Farb, who was closely involved with the ideas that led to this paper, but who declined to be listed as a coauthor. Next, my sincere thanks to Sebastian Hensel who translated Hilbert's 1927 paper into English at Benson's and my request. This paper takes place in the context of ongoing joint work with Benson Farb and Mark Kisin, and their influence permeates the perspective here. I thank Curt McMullen for helpful conversations and for extensive helpful comments and suggestions on a draft. I thank Aaron Landesman and Igor Dolgachev for helpful comments on a draft. I thank Jordan Ellenberg, Vlad Matei, Madhav Nori, Zinovy Reichstein, Daniil Rudenko and David Smyth for helpful conversations. Last, I thank the referee for many helpful comments.

\section{Tschirnhaus Complete Intersections}\label{s:tschirnci}
Given a polynomial 
\[
	p(z)=z^n+a_1z^{n-1}+\cdots+a_n=\prod(z-x_i),
\]
a {\em Tschirnhaus transformation} is a ``change of variables''
\[
	y=\sum_{j=0}^{n-1}b_jx^j.
\]
This gives a new polynomial
\[
	q(z)=z^n+c_1z^{n-1}+\cdots+c_n=\prod_i(z-y_i).
\]
We are interested in Tschirnhaus transformations such that $q(z)$ is ``better normalized'' than $p(z)$, e.g. in the sense that for some $i$,
\[
	\sum_j y_j^i=0,
\]
or more generally such that
\[
	\sum_j y_j^{i_1}=\cdots=\sum_j y_j^{i_k}=0
\]
for some $i_1,\ldots,i_k$. In this section, we study the collection of all $\b=(b_0,\ldots,b_{n-1})$ such that the above normalizations hold. These are affine varieties which we denote $\widetilde{T^n_{i_1\cdots i_k}}$, and we refer to their projectivizations $T^n_{i_1\cdots i_k}$ as {\em Tschirnhaus complete intersections}.

In this section, we introduce the varieties $T^n_{i_1\cdots i_k}$ as objects of interest in their own right, i.e. via explicit equations. We relate them to classical examples of interest, and study their geometry. In Section~\ref{s:tschirnt}, we review the classical subject of Tschirnhaus transformations for algebraic functions, and we identify the varieties $T^n_{i_1\cdots i_k}$ considered here with the spaces of ``normalized changes of variables'' described above. 
 
\subsection*{Tschirnhaus Complete Intersections via Explicit Equations}
Fix $n\ge 0$. In this section, we work over $\Z$ unless otherwise specified, so that, e.g. $\Ab^n:=\Spec(\Z[a_1,\ldots,a_n])$. For ease of reading, we adopt the following notation.
\begin{notation}\label{n:multindex}
    Denote
    \begin{equation*}
        \begin{array}{lllll}
            \a:=(a_1,\ldots,a_n)\in \Ab^n. &&&& |\kappa|:=\sum_i k_i\\
            \b:=[b_0:\cdots:b_{n-1}]\in\Pb^{n-1} &&&&  ||\kappa||:=\sum_i i\cdot k_i\\
            \kappa:=(k_0,\ldots,k_{n-1})\in\Nb^n &&&& \b^\kappa:=\prod_i b_i^{k_i}
        \end{array}
    \end{equation*}
    For $|\kappa|=i$, recall the multinomial coefficients
    \[
    \binom{i}{\kappa}:=\binom{i}{k_0,\ldots,k_{n-1}}:=\frac{i!}{k_0!\cdots k_{n-1}!}.
    \]
\end{notation}
    
We also introduce two variants of the above.
\begin{notation}\label{n:multindex'}
    \begin{align*}
    	'\b&:=[b_1:\cdots:b_{n-1}]\in\Pb^{n-2}\\
    	'\kappa&:=(k_1,\ldots,k_{n-1})\in\Nb^{n-1}\\
    	'\b'&:=[b_1:\cdots:b_{n-2}]\in\Pb^{n-3}\\
    	'\kappa'&:=(k_1,\ldots,k_{n-2})\in\Nb^{n-2}    
    \end{align*}
    Mutatis mutandis, we will also write $|'\kappa|$, $||'\kappa'||$, $\binom{i}{'\kappa}$, etc. Note that the meaning of $||-||$ depends on whether the first coordinate is the zeroth coordinate or the first coordinate.  Our notation indicates that any tuple without a $'$ preceding its label starts with a zeroth coordinate, while any tuple with a $'$ preceding its label starts with a first coordinate.
\end{notation}

We now inductively define polynomials in the $a_i$ by
\begin{align}
        p_0&:=n,\label{d:psums3}\intertext{while, for $0<k\le n$}
        p_k&:=ka_k +\sum_{i=1}^{k-1} a_{k-i}p_i,\label{d:psums1}\intertext{and for $k>n$}
        p_k&:=-\sum_{i=k-n}^{k-1} a_{k-i}p_i.\label{d:psums2}
\end{align}

\begin{remark}
    To interpret the polynomials $p_i$, let $\sigma_i$ denote the $i^{th}$ elementary symmetric polynomial in formal variables $x_1,\ldots,x_n$. If we write $a_i=(-1)^i\sigma_i$, then Newton's Identities give
    \[
        p_i=\sum_{j=1}^n x_j^i.
    \]
\end{remark}

\begin{definition}\label{d:th}
    For $i,n\ge 1$, let the $T_i^n\subset\Ab^n_\a\times\Pb^{n-1}_\b$ be the variety defined by the vanishing of the polynomial
    \begin{equation}\label{e:eti}
        \sum_{\kappa~s.t.~|\kappa|=i} \binom{i}{\kappa} p_{||\kappa||}{\bf b}^\kappa.
    \end{equation}
    Note that this polynomial is homogeneous of degree $i$ in the $\b$-coordinates. Projecting onto the first factor gives a family of degree $i$ hypersurfaces in $\Pb^{n-1}$
    \begin{equation*}
        T_i^n\to \Ab^n_\a
    \end{equation*}
    We refer to this family as the {\em $n^{th}$ Tschirnhaus hypersurface of degree $i$}. When the superscript $n$ is clear from context, we will suppress it for ease of reading.
\end{definition}

\begin{definition}\label{d:tci}
    Fix $n\ge 1$. For $1\le i_1<\ldots<i_k$, define the {\em $n^{th}$ Tschirnhaus complete intersection $T^n_{i_1\cdots i_k}$ (of multi-degree $i_1\cdots i_k$)} to be the variety defined by the vanishing of the polynomials \eqref{e:eti} for $i=i_1,\ldots,i_k$. Equivalently, define
    \begin{equation*}
        T^n_{i_1\cdots i_k}:=T^n_{i_1}\times_{\Ab^n_\a\times\Pb^{n-1}_\b}\cdots\times_{\Ab^n_\a\times\Pb^{n-1}_\b}T^n_{i_k}\to \Ab^n_\a.
    \end{equation*}
    Define the {\em $n^{th}$ reduced Tschirnhaus complete intersection $T^{n'}_{i_1\cdots i_k}$ (of multi-degree $i_1\cdots i_k$)} by
    \begin{equation*}
        T^{n'}_{i_1\cdots i_k}:= T^n_{i_1\cdots i_k}\cap \{b_0=0\}\subset\Ab^n_\a\times\Pb^{n-1}_\b.
    \end{equation*}
\end{definition}

\begin{example}\label{e:T1}
    The hyperplane $T_1(\a)\subset\Pb^{n-1}_{\b}$ is given by the equation
    \[
        nb_0+\sum_{i=1}^{n-1} p_i b_i=0
    \]
    Over $\Z[1/n]$, we have an isomorphism
    \begin{align*}
        \Ab^n_{\a}\times\Pb^{n-2}&\to^\cong T_1\\
            (\a,[b_1:\cdots:b_{n-1}])&\mapsto (\a,[-\frac{1}{n}\sum_{i=1}^{n-1} p_i b_i:b_1:\cdots:b_{n-1}]).
    \end{align*}
    Likewise, the hyperplane $T'_1(\a)\subset\Pb^{n-2}_{\b}$ is given by the equation
    \[
        \sum_{i=1}^{n-1} p_i b_i=0
    \]
    Over each locus $\{p_i\neq 0\}\subset\Ab^n_{\a}$ for $1\le i<n$, we have an isomorphism
    \begin{align*}
        \Ab^n_{\a}\times\Pb^{n-3}&\to^\cong T_1\\
            (\a,[b_1:\cdots:b_{\hat{i}}:\cdots:b_{n-2}])&\mapsto (\a,[b_1:\cdots:b_{i-1}:\frac{-1}{p_i}\sum_{j\neq i} p_jb_j:b_{i+1}:\cdots b_{n-2}]).
    \end{align*}
\end{example}

As a warm-up to Theorem~\ref{t:compint} below, we prove the following.
\begin{lemma}\label{l:quadsmooth}
    The families of quadrics $T_{12}\to\Ab^n_{\a}$ and $T'_{12}\to\Ab^n_{\a}$ are generically smooth.
\end{lemma}

\begin{remark}
    The statement of the lemma for $T_{12}$ (and most likely for $T'_{12}$) is classical, and follows from the fact that the discriminant of the quadratic form defining $T_{12}(\a)$ is equal to $\frac{1}{n}$ times the discriminant of the polynomial $x^n+a_1x^{n-1}+\cdots +a_n$ (see, e.g. \cite[p. 468-469]{Sy}). We give a different proof in order to warm-up for Theorem~\ref{t:compint}.
\end{remark}

\begin{proof}[Proof of Lemma~\ref{l:quadsmooth}]
    The quadric $T_{12}(\a)\subset\Pb^{n-2}_{\b}$ is given, in coordinates $[b_1:\cdots:b_{n-1}]$ by the equation
    \[
        -\frac{1}{n}\left(\frac{1}{n}\sum_{i=1}^{n-1} p_i b_i\right)^2+\sum_{1\le i<j\le n-1} p_{i+j} b_ib_j+\sum_{i=1}^{n-1}p_{2i}b_i^2=0.
    \]
    We now specialize to the radical pencil $x^n+a=0$, i.e. $\a=(0,\ldots,0,a)$. Then $T_{12}(a):=T_{12}(0,\ldots,a)$ is given by the equation
    \begin{align}\label{e:T12a}
        \left\{\begin{array}{cc}
                    -2na\left(\sum_{i=1}^{\frac{n-1}{2}} b_ib_{n-i}\right)=0 & n~\text{ odd}\\
                    -na\left(b_{\frac{n}{2}}^2+2\sum_{i=1}^{\frac{n}{2}-1} b_ib_{n-i}\right)=0 & n~\text{ even}
                \end{array}\right.
    \end{align}
    The partial derivatives of the defining polynomial of $T_{12}(a)$ are given by
    \begin{align*}
        \partial_{b_j}T_{12}(a)&=-2nab_{n-j}.
    \end{align*}
    We see that these vanish simultaneously if and only if $b_j=0$ for all $j$, i.e. $T_{12}(a)$ is smooth over $\Z[1/2n]$ so long as $a\neq 0$ (and thus $T_{12}\to\Ab^n_{\a}$ is generically smooth).

    We now prove $T'_{12}\to\Ab^n_{\a}$ is generically smooth. Using \eqref{d:psums1}, the hyperplane $T'_1(a)$ is given by
    \begin{equation*}
        (n-1)ab_{n-1}=0.
    \end{equation*}
    Over $\Z[1/(n-1)]$, and $a\neq 0$, we can therefore use the coordinates
    \[
        [b_1:\cdots:b_{n-2}]
    \]
    on $T_1'(a)$. In these coordinates, and abusing notation by writing the same symbol for a hypersurface and its defining polynomial, we have 
    \begin{equation*}
        T_{12}'(a)=\left\{\begin{array}{cc}
                     -2(n-1)a\left(\sum_{i=1}^{\frac{n}{2}-1} b_ib_{n-1-i}\right) & n~\text{ even}\\
                    -(n-1)a(b_{\frac{n-1}{2}}^2+\left(\sum_{i=1}^{\frac{n}{2}-1} b_ib_{n-1-i}\right)& n~\text{ odd}
                \end{array}\right.
    \end{equation*}
    The partial derivatives of $T_{12}'(a)$ are given by
    \begin{align*}
        \partial_{b_j}T'_{12}(a)&=-2(n-1)ab_{n-1-j}.
    \end{align*}
    We see that these vanish simultaneously if and only if $b_j=0$ for all $j$, i.e. $T'_{12}(a)$ is smooth over $\Z[1/2(n-1)]$ so long as $a\neq 0$ (and thus $T'_{12}\to\Ab^n_{\a}$ is generically smooth).
\end{proof}

\paragraph*{Tschirnhaus hypersurfaces as spaces of maps.}
In Section~\ref{s:tschirnt}, we explain the origin of the Tschirnhaus complete intersections in the classical study of formulas for the general degree $n$ polynomial (beginning with \cite{Ts}). For the moment, we just observe that several varieties of classical interest are closely related to $T_i^n$ for small $i,n$.

Let $\x:=(x_1,\ldots, x_n)$ be coordinates on affine $n$-space, denoted $\Ab^n_\x$. Let $\sigma_i(\x)$ denote the $i^{th}$ elementary symmetric function on the $x_i$, and consider the map
\begin{align*}
    q\colon \Ab^n_\x&\to \Ab^n_\a\\
    \x&\mapsto (-\sigma_1(\x),\ldots,(-1)^n\sigma_n(\x)).
\end{align*}
By Newton's Theorem, this map realizes $\Ab^n_\a$ as the quotient of $\Ab^n_\x$ by the permutation action of the symmetric group $S_n$ on $\Ab^n_\x$. As remarked above, Newton's Identities imply that
\[
    p_i(q(\x))=\sum_{j=1}^n x_j^i.
\]
Let $\tb:=(b_0,\ldots,b_{n-1})$ viewed as affine coordinates on $\Ab^n_{\tb}$. The relative affine cone on the pullback $T_i|_{\Ab^n_\x}\to\Ab^n_\x$ is given by
\begin{equation*}
    \widetilde{T_i|_{\Ab^n_{\x}}}:=\left\{(\x,\tb)\in\Ab^n_\x\times\Ab^n_{\tb}~|~\sum\sum_{\kappa~s.t.~|\kappa|=i} \binom{i}{\kappa} \left(\sum_{j=1}^n x_j^{||\kappa||}\right){\tb}^\kappa=0\right\}.
\end{equation*}
Consider the map
\begin{align*}
    \ev\colon\Ab^n_\x\times\Ab^n_{\tb}&\to \Ab^n_\x\\
    (\x,\tb)&\mapsto (\sum_{j=0}^{n-1}b_jx_1^j,\ldots,\sum_{j=0}^{n-1}b_jx_n^j).
\end{align*}

\begin{lemma}\label{l:Tiasmaps}
    In the notation above,
    \begin{equation*}
        \widetilde{T_i|_{\Ab^n_{\x}}}=\ev^{-1}(\{\x\in\Ab^n_\x~|~\sum_j x_j^i=0.\}).
    \end{equation*}
\end{lemma}
\begin{proof}
    We prove this by explicit computation. For $i\ge 0$, write
    \[
        p_i(\x):=\sum_{\ell=1}^n x_\ell^i.
    \]
    In particular, $p_0(x_1,\ldots,x_n)=n$. Let $\ev(\x,\tb)_\ell:=\sum_{j=0}^{n-1}b_jx_\ell^j$. By the Multinomial Theorem,
    \begin{align}
        p_i(\ev(\x,\tb))=\sum_\ell \ev(\x,\tb)_\ell^i&=\sum_\ell \left(\sum_{j=0}^{n-1}b_jx_\ell^j\right)^i\nonumber\\
        &=\sum_\ell \left(\sum_{\kappa~s.t.~|\kappa|=i}\binom{i}{\kappa}{\bf b}^\kappa x_\ell^{||\kappa||}\right)\nonumber\\
        &=\sum_{\kappa~s.t.|\kappa|=i}\binom{i}{\kappa}p_{||\kappa||}{\bf b}^\kappa\label{hyp}
    \end{align}
    where, in the final line, we use Newton's Identities to identify the power sums with the polynomials $p_{||\kappa||}$ in the $a_i$ defined in Equations~\ref{d:psums3}-\ref{d:psums2}.

    Setting the form \eqref{hyp} to 0, we obtain the hypersurface $\widetilde{T_i^n}$ as claimed.
\end{proof}

\begin{example}
    Let $S\subset\Pb^4$ be the Clebsch diagonal surface, i.e. the complete intersection
    \[
        S:=\{[x_1:\cdots:x_5]\in\Pb^4~|~\sum_{i=1}^5 x_i=\sum_{i=1}^5 x_i^3=0\}.
    \]
    Let $\widetilde{S}\subset\Ab^5_\x$ be the affine cone over $S$. Then
    \[
        \widetilde{T^5_{13}|_{\Ab^5_{\x}}}=\ev^{-1}(\widetilde{S}).
    \]
    As observed by Klein \cite[Part II, Ch. 2]{KleinIcos}, $\widetilde{T^5_{13}|_{\Ab^5_{\x}}}$ can be understood as a space of $S_5$-equivariant maps of $\Ab^5_\x\to\widetilde{S}$.
\end{example}

\begin{example}
    Let $F\subset\Pb^6$ be the symmetric Fano sextic 3-fold as in \cite{Be}, i.e. the complete intersection
    \[
        F:=\{[x_1:\cdots:x_7]\in\Pb^6~|~\sum_{i=1}^7 x_i=\sum_{i=1}^7 x_i^2=\sum_{i=1}^7 x_i^3=0\}.
    \]
    Let $\widetilde{F}\subset\Ab^7_\x$ be the affine cone over $F$. Then
    \[
        \widetilde{T^7_{123}|_{\Ab^7_{\x}}}=\ev^{-1}(\widetilde{F}).
    \]
    Though not remarked upon in \cite{Be}, the symmetric Fano sextic arises as the ``root space'' of the normal form for the general degree 7 polynomial considered by Hilbert in his 13th problem \cite{Hi1}:
     \[
    z^7+az^3+bz^2+cz+1=0.
    \].
     The variety $\widetilde{T^7_{123}|_{\Ab^7_{\x}}}$ can be understood as a space of $S_7$-equivariant maps of $\Ab^7_\x\to\widetilde{F}$, equivalently of ways of converting the general degree 7 polynomial into Hilbert's normal form.
\end{example}

\subsection*{Geometry of Tschirnhaus Complete Intersections}
We can now state our main geometric theorem.
\begin{theorem}\label{t:compint}
    Let $p$ be a prime. Let $i=p^r+1<n$ for some prime power $p^r$ with $r>0$.
    \begin{enumerate}
        \item If $p\nmid n$, the family of Tschirnhaus complete intersections
            \begin{equation*}
                T_{12i}\to \Ab^n_{\a}
            \end{equation*}
            is generically smooth (i.e. there is a Zariski open $U\subset\Ab^n_{\a}$ such that for all $\a\in U$, $T_{12i}(\a)$ is a smooth complete intersection).
        \item If $p\mid n$, the family of reduced Tschirnhaus complete intersections
            \begin{equation*}
                T_{12i}'\to\Ab^n_{\a}
            \end{equation*}
            is generically smooth.
    \end{enumerate}
\end{theorem}

Deferring the proof for a moment, let $K$ be a field of characteristic 0, now and throughout this paper.

We now record a special case of Kleiman's Bertini Theorem \cite{K}; for ease of reading, we include the proof below.
\begin{proposition}[Bertini for isotropics]\label{p:bertini}
    Let $K$ be algebraically closed. Let $X$ be a $K$-variety. Let $Q\subset\Pb^n_X$ be a smooth family of quadrics over $X$. For $k\le \lfloor\frac{n-1}{2}\rfloor$, let $\Gr(k,Q)\to X$ denote the relative Grassmannian of $k$-dimensional isotropic subspaces in $Q$, and let $\Lc\to \Gr(k,Q)$ denote the tautological bundle. Let $Y\subset\Pb^n_X$ be a smooth family of varieties over $X$ such that the family $Q\times_{\Pb^n_X} Y\to X$ is smooth over some dense open $V\subset X$. Then there exists a dense open $U\subset\Gr(k,Q)|_V$ such that the family $\Lc|_U\times_{\Pb^n_X} Y|_V\to X$ is smooth.
\end{proposition}

Combining Theorem~\ref{t:compint}, Lemma~\ref{l:quadsmooth} and Proposition~\ref{p:bertini}, we obtain the following.
\begin{corollary}\label{c:Hfix}
    Let $\Gr(T_{12})\to \Ab^n_{\a}$ denote the relative Grassmannian of maximal isotropics in the family of quadrics $T_{12}\to\Ab^n_{\a}$, and let $\Lc\to\Gr(T_{12})$ denote the tautological bundle (with similar notation for the analogous objects for $T'_{12}$). Let $p$ be a prime and let $i=p^r+1$ for some $r>0$.
    \begin{enumerate}
        \item If $p\nmid n$, there exists a dense open $V\subset \Gr(T_{12})$ such that
            \[
                \Lc|_V\times_{\Ab^n_{\a}\times\Pb^{n-1}_\b} T_{12i}\to \Ab^n_{\a}
            \]
            is smooth (i.e. for the generic polynomial, the intersection of $T_{12i}(\a)$ with a maximal isotropic in $T_{12}(\a)$ is smooth).
        \item If $p\mid n$, there exists a dense open $V\subset \Gr(T_{12}')$ such that
            \[
                \Lc|_V\times_{\Ab^n_\a\times\Pb^{n-2}_{'\b}} T_{12i}'\to \Ab^n_{\a}
            \]
            is smooth.
    \end{enumerate}
\end{corollary}
\begin{proof}
    Note that to prove the existence of an open dense $V$, it suffices to restrict all of the varieties over $\Z$ above to a geometric generic point $\Spec(K)\to\Spec(\Z)$. The result now follows immediately from Theorem~\ref{t:compint}, Lemma~\ref{l:quadsmooth} and Proposition~\ref{p:bertini}.
\end{proof}

\begin{remark}
    Corollary~\ref{c:Hfix} (for the case $p=2,i=3,n=9$) fills the gap in Hilbert's argument remarked upon by Dixmier \cite[\S8]{Di}.
\end{remark}

\begin{proof}[Proof of Proposition~\ref{p:bertini}]
    We recall Kleiman's proof \cite{K}. Consider the canonical map
    \begin{align*}
        \pr_2\colon \Lc\to Q
    \end{align*}
    (coming from the construction of $\Lc$ as an incidence variety $\Lc\subset\Gr(k,Q)\times_X Q$). Observe that this map is smooth: indeed, the relative group scheme $\O(Q)$ acts transitively over $X$ on both $\Lc$ and $Q$ (i.e. it acts transitively on fibers over $X$) and the map $\Lc\to Q$ is an $\O(Q)$-equivariant fiber bundle, with fiber at $v\in Q$ given by $\Stab_{\O(Q)}(v)/\Stab_{\O(Q)}(L,v)$, (n.b. the stabilizer of an isotropic point $v$ is a maximal parabolic, and the stabilizer of the flag $v\in L$ is a sub-parabolic).

    Let $V\subset X$ be a dense open such that $Q\times_{\Pb^n_X} Y\to X$ is smooth over $V$. Shrinking $V$ as necessary, we can assume without loss of generality that $V$ is a smooth variety over $K$ (note that we are using characteristic 0 here), and thus $(Q\times_{\Pb^n_X} Y)|_V$ is also a smooth $K$-variety. Now consider the fiber product
    \begin{equation*}
        \xymatrix{
            (\Lc\times_{\Pb^n_X} Y)|_V \ar[r]^f \ar[d]_g & (Q\times_{\Pb^n_X} Y)|_V \ar[d]^{\iota} \\
            \Lc|_V \ar[r]^{\pr_2} \ar[d]_{\pi} & Q|_V\\
            \Gr(k,Q)|_V
        }
    \end{equation*}
    The map $f$ is smooth because $\pr_2$ is smooth. Because $(Q\times_{\Pb^n_X} Y)|_V$ is a smooth $K$-variety, the $K$-variety $(\Lc\times_{\Pb^n_X} Y)|_V$ is smooth. We therefore have a dominant map of smooth $K$-varieties
    \[
        q=\pi\circ g\colon (\Lc\times_{\Pb^n_X} Y)|_V\to\Gr(k,Q)|_V.
    \]
    By generic smoothness (e.g. \cite[Corollary III.10.7]{Har}), there exists a nonempty open subset $U\subset\Gr(k,Q)|_V$ such that $q\colon (\Lc|_U\times_{\Pb^n_X} Y|_V)\to U$ is smooth, and thus the composite $(\Lc|_U\times_{\Pb^n_X} Y|_V)\to U\to V$ is smooth as well.
\end{proof}

We now prove Theorem~\ref{t:compint}.
\begin{proof}[Proof of Theorem~\ref{t:compint}]
    We prove the two cases separately, via parallel arguments. As in the proof of Lemma~\ref{l:quadsmooth}, if it will not cause confusion, we will abuse notation by writing the same symbol to denote a complete intersection and its defining polynomials.
    \newline
    \medskip
    \noindent
    {\bf Case 1: $p\nmid n$.} The complete intersection $T_{12i}({\bf a})$ is smooth if and only if the $3\times n$ matrix
    \[
        \left(\begin{array}{ccc}
            \partial_{b_1} T_1({\bf a}) & \cdots & \partial_{b_{n-1}} T_1({\bf a}) \\
            \partial_{b_1} T_2({\bf a}) & \cdots & \partial_{b_{n-1}} T_2({\bf a}) \\
            \partial_{b_1} T_i({\bf a}) & \cdots & \partial_{b_{n-1}} T_i({\bf a})
            \end{array}\right)
    \]
    has full rank for all ${\bf b}\in T_{12i}({\bf a})$. Choosing coordinates on $T_1$, we can equivalently check whether the $2\times (n-1)$ matrix given by the partials of $T_{12}$ and $T_{1i}$ has rank 2 for all ${\bf b}\in T_{12i}({\bf a})$. To show generic smoothness, it suffices to find a single ${\bf a}$ for which this holds. Further, because the matrix above is defined over $\Zb$, to show it is nonsingular in characteristic 0, it suffices to find a prime $p$ for which its reduction mod $p$ is nonsingular.

    We specialize to the locus of radical polynomials, i.e. those of the form
    \[
        p(x)=x^n+a
    \]
    i.e. ${\bf a}=(0,\ldots,0,a)$. It suffices to show there exists $a$ such that $T_{12i}(a):=T_{12i}(0,\ldots,a)$ is smooth. Note that, restricting to $x^n+a$, the hyperplane $T_1(a)$ is given by
    \[
        nb_0=0.
    \]
    We can therefore use the coordinates
    \[
        [b_1:\cdots:b_{n-1}]
    \]
    on $T_1(a)$ as above. As in \eqref{e:T12a}, the form $T_{12}(a)$ is given in these coordinates by
    \begin{equation*}
        T_{12}(a)=\left\{\begin{array}{cc}
                    -2na\left(\sum_{i=1}^{\frac{n-1}{2}} b_ib_{n-i}\right) & n~\text{ odd}\\
                    -na\left(b_{\frac{n}{2}}^2+2\sum_{i=1}^{\frac{n}{2}-1} b_ib_{n-i}\right) & n~\text{ even}
                \end{array}\right.
    \end{equation*}
    and the partial derivatives are given by
    \begin{align*}
        \partial_{b_j}T_{12}(a)&=-2nab_{n-j}.
    \end{align*}
    Similarly, using Notation~\ref{n:multindex'}, the form $T_{1i}(a)$ is given by
    \begin{equation*}
        T_{1i}(a)=n\cdot\left(\sum_{\ell=1}^{i-1}(-1)^\ell a^\ell \left(\sum_{'\kappa~s.t.|'\kappa|=i,||'\kappa||=\ell n} \binom{i}{'\kappa}{\bf 'b}^{'\kappa}\right)\right)
    \end{equation*}
    The partial derivatives of $T_{1i}(a)$ are given by
    \begin{align*}
        \partial_{b_j}T_{1i}(a)&=in\cdot\left(\sum_{\ell=1}^{i-1}(-1)^\ell a^\ell\left(\sum_{'\kappa~s.t.~|'\kappa|=i-1,||'\kappa||+j=\ell n} \binom{i-1}{'\kappa}{\bf 'b}^{'\kappa}\right)\right).
    \end{align*}
    Define
    \begin{align*}
        T_{j,12}(a)&:=ab_{n-j}\\
        T_{j,1i}(a)&:=\sum_{\ell=1}^{i-1}(-1)^\ell a^\ell\left(\sum_{'\kappa~s.t.~|'\kappa|=i-1,||'\kappa||+j=\ell n} \binom{i-1}{'\kappa}{\bf 'b}^{'\kappa}\right).
    \end{align*}
    Then, in characteristic 0, the matrix
    \[
        \left(\begin{array}{ccc}
            \partial_{b_1} T_{12}(a) & \cdots & \partial_{b_{n-1}} T_{12}(a) \\
            \partial_{b_1} T_{1i}(a) & \cdots & \partial_{b_{n-1}} T_{1i}(a)
            \end{array}\right)
    \]
    is singular if and only if the matrix
    \[
        \left(\begin{array}{ccc}
            T_{1,12}(a) & \cdots & T_{n-1,12}(a) \\
            T_{1,1i}(a) & \cdots & T_{n-1,1i}(a)
            \end{array}\right)
    \]
    is singular. Because this matrix is defined over $\Zb[a]$, to show that it is generically nonsingular in characteristic 0, we can reduce mod $p$ and find some $a\in\overline{\F_p}$ for which it is nonsingular.

    Let $\overline{T_{j,12}}(a)$ and $\overline{T_{j,1i}}(a)$ denote the reduction of the above forms mod $p$.

    Recall 
    that Legendre's formula implies that a prime $p$ divides 
   all the multinomial coefficients $\{\binom{\ell}{k_1,\ldots,k_m}~|~k_j<\ell\text{ for all }j\}$ if and only if $\ell=p^r$. Therefore, reducing the forms $T_{j,1i}(a)$ mod $p$, and using $i-1=p^r$, Legendre's formula implies that .
    \begin{align}
        \overline{T_{j,1i}}(a)=\sum_{\ell=1}^{i-1}(-1)^\ell a^\ell\left(\sum_{1\le \nu\le n-1,p^r\nu+j=\ell n} b_\nu^{p^r}\right)
    \end{align}
	(n.b. as we remark just below, only one term in the above sum is nonzero). Now, because $p\nmid n$, $p^r\in(\Z/n\Z)^\times$. Therefore, multiplication by $p^{-r}$ determines a permutation of $\{1,\ldots,n-1\}=\Z/n\Z-\{0\}$, which we denote by
    \[
        \nu(j):=p^{-r}\cdot j\in \Z/n\Z-\{0\}=\{1,\ldots,n-1\}.
    \]
    In this notation, we have
    \begin{align*}
        \overline{T_{j,12}}(a)&=ab_{-j}\\
        \overline{T_{j,1i}}(a)&=(-a)^{\frac{p^r\nu(-j)+j}{n}}b_{\nu(-j)}^{p^r}
    \end{align*}
    where $\pm j$ and $\nu(\pm j)$ denote the corresponding elements of $\{1,\ldots,n-1\}$. Now, multiplication by $p^{-r}$ on $\Z/n\Z-\{0\}$ generates a cyclic group, and so a partition of $\{1,\ldots,n-1\}$ into $m$ orbits $\Oc_\alpha$ of size $s_\alpha$. Let $j_\alpha$ denote the least element of the orbit $\Oc_\alpha$. For ease of notation, denote
    \begin{equation*}
    	\epsilon_\alpha(t):=\frac{p^r\nu^t(j_\alpha)+n-\nu^{t-1}(j_\alpha)}{n}.
    \end{equation*}
	Reorder the columns of the matrix we are considering so that it is of the form
    \begin{align}\label{e:Mmat}
        M:=\left(\begin{array}{ccc}
            M_1 & \cdots & M_m
        \end{array}\right)
    \end{align}
    where each $M_\alpha$ denotes the $2\times s_\alpha$ matrix
    \begin{align*}
    	M_\alpha:=\left(\begin{array}{cccc}
    		ab_{j_\alpha} & ab_{\nu(j_\alpha)} & \cdots & ab_{\nu^{s_\alpha-1}(j_\alpha)}\\
    		(-a)^{\epsilon_\alpha(1)}b_{\nu(j_\alpha)}^{p^r} & (-a)^{\epsilon_\alpha(2)}b_{\nu^2(j_\alpha)}^{p^r} & \cdots & (-a)^{\epsilon_\alpha(s_\alpha)}b_{j_\alpha}^{p^r}
    	\end{array}\right).
    \end{align*}
Note that, by construction, for each $j$, all monomials containing $b_j$ appear in precisely one $M_\alpha$. 
    

Now the matrix \eqref{e:Mmat} is singular at $\b\in\Pb^{n-2}$ and $a\in\Fpbar$ if and only if its two rows are linearly dependent. Equivalently, there exists $\lambda\in\Fpbar^\times$ such that for all $\alpha$ and $0\le t\le s_\alpha-1$
    \begin{equation}\label{e:intcycle}
        ab_{\nu^t(j_\alpha)}=\lambda (-a)^{\epsilon_\alpha(t+1)}b_{\nu^{t+1}(j_\alpha)}^{p^r}.
    \end{equation}
    Restrict to $a\in\Fpbar^\times$. Then, by induction on $t$, we obtain that for all $j\in\Oc_\alpha$
    \begin{equation*}
        b_j=(-\lambda)^{\sum_{t=1}^{s_\alpha}p^{(t-1)r}}(-a)^{\sum_{t=1}^{s_\alpha}p^{(t-1)r}(\epsilon_\alpha(t)-1)}b_j^{p^{s_\alpha r}}.
    \end{equation*}
    Therefore, for any $b_j\neq 0$ for $j\in \Oc_\alpha$ (and such a $j$ and $\alpha$ must exist since ${\bf b}\in\Pb^{n-2}$), we have
    \begin{align*}
        b_j^{p^{s_\alpha r}-1}&=(-\lambda)^{-\sum_{t=1}^{s_\alpha}p^{(t-1)r}}(-a)^{-\sum_{t=1}^{s_\alpha}p^{(t-1)r}(\epsilon_\alpha(t)-1)}\\
        &=:c_\alpha(a)
    \end{align*}
    But, if $j=\nu^t(j_\alpha)$, then by Equation~\eqref{e:intcycle},
    \begin{equation*}
        c_\alpha(a)=b_j^{p^{s_\alpha r}-1}=(-\lambda (-a)^{\epsilon_\alpha(t+1)-1})^{p^{s_\alpha r}-1}c_\alpha(a)^{p^r}.
    \end{equation*}
    Expanding the definition of $c_\alpha(a)$ in terms of $\lambda$ and $a$, we obtain
    \begin{align*}
    (-\lambda)^{-\sum_{t=1}^{s_\alpha}p^{(t-1)r}}&(-a)^{-\sum_{t=1}^{s_\alpha}p^{(t-1)r}(\epsilon_\alpha(t)-1)}\\
        &=(-\lambda)^{p^{s_\alpha r}-1-\sum_{t=1}^{s_\alpha}p^{tr}}(-a)^{(p^{s_\alpha r}-1)(\epsilon_\alpha(t+1)-1)-\sum_{t=1}^{s_\alpha}p^{tr}(\epsilon_\alpha(t)-1)}\\
        &=(-\lambda)^{-\sum_{t=1}^{s_\alpha}p^{(t-1)r}}(-a)^{(p^{s_\alpha r}-1)(\epsilon_\alpha(t+1)-1)-\sum_{t=1}^{s_\alpha}p^{tr}(\epsilon_\alpha(t)-1)}.\intertext{Therefore, for all $0\le t\le s_\alpha$}
        1&=(-a)^{(p^{s_\alpha r}-1)(\epsilon_\alpha(t+1)-1)-\sum_{t=1}^{s_\alpha}p^{(t-1)r}(p^r-1)(\epsilon_\alpha(t)-1)}
    \end{align*}
    In particular,
    \begin{equation}\label{e:contra1}
        a^{2(p^{s_\alpha r}-1)(\epsilon_\alpha(t+1)-1)-\sum_{t=1}^{s_\alpha}p^{(t-1)r}(p^r-1)(\epsilon_\alpha(t)-1)}=1.
    \end{equation}
    But, $s_\alpha,\epsilon_\alpha(t),p,r\in\Nb$ are fixed once and for all by our choice of $p$ and $n$. In particular, there exists $N\in\Nb$ such that
    \[
        N>\max_\alpha |2(p^{s_\alpha r}-1)(\epsilon_\alpha(t+1)-1)-\sum_{t=1}^{s_\alpha}p^{(t-1)r}(p^r-1)(\epsilon_\alpha(t)-1)|.
    \]
    But, then for any primitive $N^{th}$ root of unity $a\in\Fpbar$, Equation~\ref{e:contra1} is never satisfied. Therefore, the matrix $M(a)=\left( M_1(a) \cdots M_m(a)\right)$ of \eqref{e:Mmat} has full rank for all ${\bf 'b}\in\Pb^{n-2}$ as claimed.
    \medskip
    \newline
    \noindent
    {\bf Case 2: $p\mid n$.}
    This case is similar. We specialize to the pencil $x^n+ax=0$, i.e. ${\bf a}=(0,\ldots,0,a,0)$. It suffices to show there exists $a$ such that $T_{12i}'(a):=T_{12i}'(0,\ldots,a,0)$ is smooth.

    As noted in the proof of Lemma~\ref{l:quadsmooth}, over $\Z[1/(n-1)]$, and $a\neq 0$, we can use the coordinates
    \[
        [b_1:\cdots:b_{n-2}]
    \]
    on $T_1'(a)$. We follow Notation~\ref{n:multindex'}. In these coordinates and this notation, the partial derivatives of $T_{12}'(a)$ are given by
    \begin{align*}
        \partial_{b_j}T_{12}(a)&=-2(n-1)ab_{n-1-j}
    \end{align*}
    (as noted in the proof of Lemma~\ref{l:quadsmooth}). Similarly, we have
    \begin{align*}
        T_{1i}'(a)=&(n-1)\cdot\left(\sum_{\ell=1}^{i-1}(-a)^\ell\sum_{'\kappa'~s.t.~|'\kappa'|=i,||'\kappa'||=\ell (n-1)}\binom{i}{'\kappa'}~'\b'^{('\kappa')}\right)\\
        \partial_{b_j}T_{1i}'(a)=&i(n-1)\cdot\left(\sum_{\ell=1}^{i-1}(-a)^\ell\sum_{'\kappa'~s.t.~|'\kappa'|=i-1,||'\kappa'||+j=\ell(n-1)}\binom{i-1}{'\kappa'}~'\b'^{('\kappa')}\right)
    \end{align*}
    Define
    \begin{align*}
        T_{j,12}'(a)&:=ab_{n-1-j}\\
        T_{j,1i}'(a)&:=\sum_{\ell=1}^{i-1}(-a)^\ell\sum_{'\kappa'~s.t.~|'\kappa'|=i-1,||'\kappa'||+j=\ell(n-1)}\binom{i-1}{'\kappa'}~'\b'^{('\kappa')}\\
    \end{align*}
    Just as in Case 1, the matrix
    \[
        \left(\begin{array}{ccc}
            \partial_{b_1} T_{12}(a) & \cdots & \partial_{b_{n-2}} T_{12}(a) \\
            \partial_{b_1} T_{1i}(a) & \cdots & \partial_{b_{n-2}} T_{1i}(a)
            \end{array}\right)
    \]
    is everywhere nonsingular in characteristic 0 for some $a$ if and only if the matrix
    \[
        \left(\begin{array}{ccc}
            T_{1,12}'(a) & \cdots & T_{n-2,12}'(a) \\
            T_{1,1i}'(a) & \cdots & T_{n-2,1i}'(a)
            \end{array}\right)
    \]
    is everywhere nonsingular for some $a$. We now reduce this matrix mod $p$. Because $i=p^r+1$, the mod $p$ reduction of $T_{j,1i}'(a)$ is given by
    \begin{align*}
        \overline{T_{j,1i}'}(a)=\sum_{\ell=1}^{i-1}(-a)^\ell\sum_{'\kappa'~s.t.~|'\kappa'|=i-1,||'\kappa'||+j=\ell(n-1)}\binom{i-1}{'\kappa'}~'\b'^{('\kappa')}
    \end{align*}
    In particular, because $i-1=p^r$, and $p^r\in(\Z/(n-1)\Z)^\times$, the same arguments as above allow us to define a permutation $\nu\circlearrowleft\{1,\ldots,n-2\}=(\Z/(n-1)\Z)-\{0\}$ by
    \[
        \nu(j)=p^{-r}j\in(\Z/(n-1)\Z)-\{0\}=\{1,\ldots,n-2\}.
    \]
    Using $\nu$, we have
    \begin{align*}
        \overline{T_{j,1i}'}(a)=(-a)^{\frac{p^r\nu(j)+j}{n-1}}b_{\nu(j)}^{p^r}.
    \end{align*}
    {\em Mutatis mutandis}, we now complete the argument by the same reasoning as for Case 1.
\end{proof}

\begin{remark}
    A similar argument shows that the Tschirnhaus hypersurface $T_i\to\Ab^n_{\a}$ itself is generically smooth for $i=p^r+1$ and $r\ge 0$. More generally, we see no reason not to expect this, as well as Theorem~\ref{t:compint}, to hold without restriction on $i<n$. In principle, this comes down to checking whether an appropriate discriminant identically vanishes on $T_i$ (resp. $T_{12i}$), i.e. checking a polynomial condition on the form defining $T_i$. However, this discriminant is a polynomial of degree $(n-1)(d-1)^{n-1}$ in the coefficients of the form, and the number of terms in this polynomial grows so quickly as to make direct computation impossible except for very small $d$ and $n$.
\end{remark}

\section{Algebraic Functions and Tschirnhaus Transformations}\label{s:tschirnt}
In this section, we recall the theory of Tschirnhaus transformations of algebraic functions and relate this to the Tschirnhaus complete intersections studied above.

Let $X$ be an irreducible $K$-variety. We write $K(X)$ for the rational functions on $X$. More generally, for a (not necessarily reducible) $K$-variety $Y$ with irreducible components $\{Y_i\}$, let $K(Y):=\prod_i K(Y_i)$.

Recall that an {\em algebraic function} $\Phi$ on $X$ is a finite rational correspondence $X\dashrightarrow^{1:n}\Ab^1$, i.e. $\Phi$ is given by a span
\begin{equation*}
    \xymatrix{
        E_\Phi \ar[r]^z \ar[d]_\pi & \Ab^1 \\
        X
    }
\end{equation*}
where $\pi$ is a dominant, quasi-finite map and $z$ is a regular function. We say $\Phi$ is {\em irreducible} if $E_\Phi$ is an irreducible $K$-variety and $z$ is a primitive element of the finite field extension $K(E_\Phi)/K(X)$. As a bridge to the classical literature, we will also denote $K(E_\Phi)$ as $K(X)(\Phi)$ to emphasize that $K(E_\Phi)$ is obtained from the field $K(X)$ by adjoining the values of $\Phi$.

Let $\Mon(\Phi)$ denote the {\em monodromy group} of $\Phi$, equivalently the Galois group of the normal closure of $K(X)(\Phi)/K(X)$. Let
\begin{equation*}
    m_{\Phi}(z):=z^n+a_1z^{n-1}+\ldots+a_n
\end{equation*}
denote the minimal polynomial of $z$, where the $a_i\in K(X)$ (i.e. $m_{\Phi}(z)$ is the monic generator of the ideal of $K(X)[z]$ corresponding to the extension $K(X)(\Phi)$). A classical perspective describes $\Phi$ as the assignment
\begin{equation}\label{algfun}
    x\mapsto \{z\in\bar{K}~|~m_{\Phi(x)}(z)=z^n+a_1(x)z^{n-1}+\ldots+a_n(x)=0\}.
\end{equation}

For any field extension $K(X)\into L$, write
\begin{equation*}
    L(\Phi):=L\otimes_{K(X)} K(X)(\Phi).
\end{equation*}
Note that since $\{1,z,\ldots,z^{n-1}\}$ is a basis for $K(X)(\Phi)$ over $K(X)$, it is also a basis for $L(\Phi)$ over $L$. Given this, for each $w\in L(\Phi)$, there exist unique $b_0,\ldots,b_{n-1}\in L$ such that
\[
    w=\sum_{i=0}^{n-1} b_iz^i.
\]
Moreover, $\tb=(b_0,\ldots,b_{n-1})\in L^n$ determines an $L$-linear transformation
\[
    T_{\tb}:L(\Phi)\to L(\Phi)
\]
given by (extending $L$-linearly) the assignment $T_{\tb}(z^j):=w^j$ for each $0\leq j\leq n-1$. Note that $T_{\tb}$ is an automorphism if and only if $w$ is a primitive element of the extension $L(\Phi)/L$.

\begin{definition}
    Let $X$ be an irreducible $K$-variety. Let $\Phi$ be an irreducible algebraic function on $X$ with primitive element $z\in K(X)(\Phi)$. A {\em Tschirnhaus transformation} $T$ of $\Phi$ is a $\overline{K(X)}$-linear automorphism
    \begin{equation*}
        T\colon \overline{K(X)}(\Phi)\to \overline{K(X)}(\Phi).
    \end{equation*}
    of the form
    \begin{equation*}
        z^j\mapsto w^j=\left(\sum_{i=0}^{n-1}b_i z^i\right)^j
    \end{equation*}
    for $b_0,\ldots,b_{n-1}\in \overline{K(X)}$. We say the transformation is {\em rational} over $X$ if $b_0,\ldots,b_{n-1}\in K(X)$. More generally, we say it is {\em rational} over $L/K(X)$ if all $b_i\in L$.
\end{definition}

Picking an integral model $Y\to X$ for $K(X)(\tb)/K(X)$, (i.e. a map of $K$-varieties $Y\to X$ and an isomorphism $K(Y)\cong K(X)(\tb)$ as extensions of $K(X)$), we denote by $T(\Phi)$ the algebraic function on $Y$ determined by the primitive element $w\in K(Y)(\Phi)$.

Now let $\Phi$ be an algebraic function as above, and $T$ a Tschirnhaus transformation of $\Phi$. Let $w=T(z)$, and let the minimal polynomial of multiplication by $w$ on $\overline{K(X)}(\Phi)$ be given by
\begin{equation*}
    m_{T(\Phi)}(w):=w^n+c_1w^{n-1}+\ldots c_n
\end{equation*}
where $c_i\in L=K(Y)$. The algebraic function $T(\Phi)$ on $Y$ is given by the assignment
\begin{equation*}
    y\mapsto \{z\in\overline{K(X)}~|~m_{T(\Phi)(y)}(z)=z^n+c_1(y)z^{n-1}+\ldots+c_n(y)=0\}.
\end{equation*}

Recall that $\Ab^n_X:=X\times_{\Spec(K)} \Ab^n_K$, viewed as a variety over $X$.
\begin{lemma}
    Let $X$ be irreducible, and let $\Phi$ be an irreducible, generically $n$-valued algebraic function on $X$. Then there is an open subvariety
    \[
        \Tc_\Phi\subset \Ab^n_X,
    \]
    such that for all finite extensions $L/K(X)$, $\Tc_\Phi(L)$ is the set of Tschirnhaus transformations of $\Phi$ which are rational over $L$. In particular, the map
    \[
        \xymatrix{
            \Tc_\Phi \ar@{^{(}->}[r] \ar[dr] & \Ab^n_X \ar[d] \\
                & X
            }
    \]
    is smooth. Equivalently the parameter space of Tschirnhaus transformations $\Tc_\Phi\to X$ is smooth over $X$.
\end{lemma}
\begin{proof}
    We begin by constructing the variety $\Tc_\Phi$. Denote the set of $\overline{K(X)}$-rational Tschirnhaus transformations of $\Phi$ by $\Tc_\Phi(\overline{K(X)})$. We will show that this embeds as an explicit Zariski open subset of $\overline{K(X)}^n=\Ab^n_X(\overline{K(X)})$, and that its complement is defined over $K(X)$; we thus conclude that $\Tc_\Phi(\overline{K(X)})$ is the set of geometric generic points of a variety $\Tc_\Phi\subset\Ab^n_X$.

    Let $z\in K(X)(\Phi)$ be the primitive element determined by $\Phi$. Given $\tb\in\overline{K(X)}^n$, we have a $\overline{K(X)}$-linear endomorphism
    \begin{align*}
        T_{\tb}\colon\overline{K(X)}(\Phi)&\to \overline{K(X)}(\Phi)\intertext{given by}
        z^j&\mapsto \left(\sum_{i=0}^{n-1} b_i z^i\right)^j.
    \end{align*}
    Moreover, the assignment $\tb\mapsto T_{\tb}$ defines a $\Gal(\overline{K(X)}/K(X))$-equivariant map
    \begin{equation*}
        T\colon\Ab^n(\overline{K(X)})\to\End_{\overline{K(X)}}(\overline{K(X)}(\Phi))\cong\Ab^{n^2}(\overline{K(X)}).
    \end{equation*}
    By definition, $\Tc_\Phi(\overline{K(X)})$ is in bijection with the set
    \[
        \{\tb\in\overline{K(X)}^n~|~ T_{\tb}\in \Aut_{\overline{K(X)}}(\overline{K(X)}(\Phi))\}
    \]
    i.e.
    \[
        \Tc_\Phi(\overline{K(X)})=T^{-1}(\Aut_{\overline{K(X)}}(\overline{K(X)}(\Phi))).
    \]
    Since $\Aut_{\overline{K(X)}}(\overline{K(X)}(\Phi))$ is the pullback to $\overline{K(X)}$ of an open subvariety of $\Ab^{n^2}_\Z$ (i.e. the locus $\{\det\neq 0\})$) and $T$ is defined over $K(X)$, we conclude that $\Tc_\Phi(\overline{K(X)})\subset\Ab^n(\overline{K(X)}$ is Zariski open and defined over $K(X)$ as claimed. The remaining claims follow by direct inspection.
\end{proof}

\begin{corollary}
    Let $\Phi$ be an irreducible $n$-valued algebraic function on $X$ such that $K(X)(\Phi)/K(X)$ has no intermediate subfields. Let $\Ab^n_X$ be given coordinates $(b_0,\ldots,b_{n-1})$ as above, and let $\Ab^1_{X,0}\subset\Ab^n_X$ denote the $b_0$-axis. Then
    \[
        \Tc_\Phi=\Ab^n_X-\Ab^1_{X,0}.
    \]
\end{corollary}
\begin{proof}
    Because $K(X)(\Phi)/K(X)$ has no intermediate subfields, $y\in K(X)(\Phi)$ is a primitive element if and only if $y\notin K(X)$, i.e. if and only if $y$ is of the form $y=\sum_{i=0}^{n-1} b_iz^i$ with $b_i\neq 0$ for some $i>0$.
\end{proof}

\begin{example}\label{ex:poly}
    Let $X=\Ab^n_\a$, viewed as the parameter space for monic, degree $n$ polynomials (parametrized by their coefficients $\a:=(a_1,\ldots,a_n)$). Let $P_n$ be the general degree $n$ polynomial, i.e.
    \[
        m_{P_n}(z)=z^n+a_1z^{n-1}+\ldots+a_n.
    \]
    Then the degree $n$ extension $K(\Ab^n_\a)(P_n)/K(\Ab^n_\a)$ has no intermediate subfields, because it corresponds to the maximal subgroup $S_{n-1}\subset S_n=\Mon(P_n)$. In particular, the space of Tschirnhaus transformations of the general degree $n$ polynomial is given by
    \begin{align*}
        \Tc_{P_n}&=\Ab^n_X-\Ab^1_{X,0}\\
        &:=\Ab^n_{\tb}\times\Ab^n_\a -\Ab^1_{b_0}\times\Ab^n_a\\
        &=(\Ab^n_{\tb}-\Ab^1_{b_0})\times\Ab^n_\a.
    \end{align*}
\end{example}

Now let $\Phi$ be an irreducible algebraic function on $X$, and let $T$ be a Tschirnhaus transformation of $\Phi$ as above, with minimal polynomial
\begin{equation*}
    m_{T(\Phi)}(y):=y^n+c_1y^{n-1}+\ldots c_n
\end{equation*}
Observe that the assignment
\begin{equation*}
    x\mapsto (c_1(x),\ldots,c_n(x))
\end{equation*}
determines a rational map
\begin{equation*}
    X\dashrightarrow \Ab^n
\end{equation*}
which fits into a pullback square
\begin{equation*}
    \xymatrix{
        E_\Phi \ar@{-->}[r] \ar[d]_\pi & E_{P_n} \ar[d] \\
        X \ar@{-->}[r] & \Ab^n
    }
\end{equation*}
In particular, the Tschirnhaus transformation $T$ transforms $\Phi$ into a function of $d=\dim(\Image(X\dashrightarrow \Ab^n))$ variables.

We now study loci of interest in the space of Tschirnhaus transformations. The basic observation (essentially going back to Tschirnhaus \cite{Ts}) is as follows. First, the collection of $n$-valued algebraic functions on $X$ is given by $\Ab^n_X$, where $\a=(a_1,\ldots,a_n)\in\Ab^n_X$ corresponds to the function $\Phi_\a$ of \eqref{algfun}, i.e. the function
\begin{equation*}
    x\mapsto \{z\in\bar{K}~|~m_{\Phi_\a(x)}(z)=z^n+a_1(x)z^{n-1}+\ldots+a_n(x)=0\}.
\end{equation*}
Next, the assignment $(\Phi_\a,\tb)\mapsto T_{\tb}(\Phi_\a)$ determines an ``evaluation'' map
\begin{align*}
    \Ab^n_{X,\a}\times \Ab^n_{X,\tb}&\to^{\ev} \Ab^n_{X,\a}\\
    (\a,\tb)&\mapsto T_{\tb}(\a)\nonumber
\end{align*}
(where we write $(-)_\a$ and $(-)_{\tb}$ to distinguish the different roles of the $\a$ and $\tb$ coordinates). The coordinates of $T_{\tb}(\a)$ can be computed explicitly as follows.  By definition, $\tb\in \Ab^n_X$ corresponds to the assignment
\begin{equation*}
    z\mapsto \sum_{i=0}^{n-1} b_iz^i=y
\end{equation*}
for $z$ a value of $\Phi_\a$. Passing to a Galois closure of $K(X)(\Phi)$, the transformation $T$ maps the roots $z_i$ of $m_\Phi$ to $y_i$ given by
\[
    y_i=\sum_{j=0}^{n-1} b_jz_i^j.
\]
In particular, the polynomial $m_{T(\Phi)}$ is given by
\[
    m_{T(\Phi)}(y)=\prod_{i=1}^n(y-y_i).
\]
i.e. the coordinates of $T_{\Phi}$ are obtained (up to sign) by expanding the elementary symmetric polynomials in the $y_i$ as polynomials in $\b$ with coefficients given by polynomials in the coordinates $\a$. In particular, the $j^{th}$ coefficient is a homogeneous polynomial of total degree $j$ in the coordinates $\tb$.

As a result, every Zariski closed subvariety $Z\subset\Ab^n_{X,\a}$ determines a Zariski closed subvariety
\[
    \ev^{-1}(Z)\subset\Ab^n_{X,\a}\times\Ab^n_{X,\tb},
\]
Specializing to a particular algebraic function $\Phi$, and its space of Tschirnhaus transformations $\Tc_\Phi\subset\Ab^n_{X,\tb}$, we obtain a Zariski closed subvariety (concretely $\Tc_\Phi\cap\ev^{-1}(Z)$), which, by abuse of notation, we denote again by
\[
    \ev^{-1}(Z)\subset\Tc_\Phi.
\]
By construction, this subvariety parametrizes Tschirnhaus transformations of $\Phi$ such that $T(\Phi)$ (or more precisely, the coefficients of its minimal polynomial) lie in $Z\subset\Ab^n_{X,\a}$.

We can now make contact with the Tschirnhaus complete intersections introduced in Section~\ref{s:tschirnci}. For $1\le i_1<\ldots<i_k$, define
\[
    Z_{i_1\cdots i_k}:=\{\a\in \Ab^n_{\a}~|~p_{i_1}(\a)=\cdots=p_{i_k}(\a)=0\}
\]
where the $p_i$s are as in Section~\ref{s:tschirnci}.
\begin{definition}
    Let $n> 0$. For $1\le i_1<\ldots<i_k$, define the {\em affine Tschirnhaus complete intersection} $\tilde{T}_{i_1\cdots i_k}(P_n)$ to be
     \[
        \tilde{T}_{i_1\cdots i_k}(P_n):=\ev^{-1}(Z_{i_1\cdots i_j})\subset \Tc_{P_n}\subset(\Ab^n_{\tb}-\Ab^1_{b_0=0})\times\Ab^n_{\a}.
    \]
    Projecting onto $\Ab^n_\a$ gives the family $\tilde{T}_{i_1\cdots i_k}(P_n)\to \Ab^n_\a$.

    Similarly, define the {\em Tschirnhaus complete intersection}
    \[
        T_{i_1\cdots i_k}(P_n)\subset (\Pb^{n-1}_{\b}-\{[1:0:\cdots:0]\})\times \Ab^n_{\a}
    \]
    to be the (fiberwise) projectivization of the family $\tilde{T}_{i_1\cdots i_k}(P_n)\to \Ab^n_{\a}$.

    Define the {\em reduced affine Tschirnhaus complete intersection} by
    \[
        \tilde{T}'_{i_1\cdots i_k}(P_n):=T_{i_1\cdots i_k}\cap \{b_0=0\}.
    \]
    Similarly, define the {\em reduced Tschirnhaus complete intersection}
    \[
        T'_{i_1\cdots i_k}(P_n)\subset\Pb^{n-2}_{'\b}\times\Ab^n_{\a}
    \]
    to be the (fiberwise) projectivization of the family $\tilde{T}'_{i_1\cdots i_k}\to \Ab^n_{\a}$.
\end{definition}

Lemma~\ref{l:Tiasmaps} can now be equivalently restated as follows.
\begin{lemma}\label{l:Tiagree}
    For all $n$ and all $1\le i_1\le \cdots\le i_k$, we have
    \begin{align*}
        T_{i_1\cdots i_k}(P_n)=T^n_{i_1\cdots i_k}
    \end{align*}
    as subvarieties of $\Ab^n_\a\times\Pb^{n-1}_\b$, where the right hand side denotes the Tschirnhaus complete intersection of Definition~\ref{d:tci}.

    Similarly, we have
    \begin{align*}
        T'_{i_1\cdots i_k}(P_n)=T^{n'}_{i_1\cdots i_k}
    \end{align*}
    as subvarieties of $\Ab^n_\a\times\Pb^{n-2}_{'\b}$.
\end{lemma}

\section{The Resolvent Degree of a Dominant Map}\label{s:rd}
Recall the following (see \cite{Br,AS,FW}).
\begin{definition}[{\bf Resolvent degree}]
\label{d:rdfdom}
    Let $Y\to X$ be a generically finite dominant map of $K$-varieties. Its {\em resolvent degree} $\RD(Y\to X)$ is the minimum $d$ for which there exists a dense Zariski open $U\subset X$ and a tower of generically finite dominant maps
    \[
        E_r\to\cdots\to E_1\to E_0=U
    \]
    such that $E_r\to U$ factors through a dominant map $E_r\to Y$ and such that for each $i\ge 0$, there exists a pullback diagram
    \begin{equation*}
        \xymatrix{
            E_i \ar[r] \ar[d] & \tilde{Z}_i \ar[d] \\
            E_{i-1} \ar[r] & Z_i
        }
    \end{equation*}
    where $\tilde{Z}_i\to Z_i$ is a generically finite dominant map with $\dim(Z_i)\le d$.
\end{definition}

\begin{example}
    Consider the space $\Ab^n_\a$ of monic degree $n$-polynomials. This has a canonical $n$-sheeted branched cover $E_{P_n}\to\Ab^n_\a$ where $E_{P_n}$ is the space of monic degree $n$ polynomials with a choice of root, and the map forgets the root.  By definition
    \[
        \RD(n):=\RD(E_{P_n}\to\Ab^n_\a).
    \]
\end{example}
We now extend the notion of resolvent degree to general dominant maps. We adopt the following convention to avoid pathologies.
\begin{convention}
    By a {\em dominant} map, we mean a map $Y\to X$ that is both dominant, and is such that every irreducible component of $Y$ maps dominantly onto some irreducible component of $X$.
\end{convention}

\begin{definition}[{\bf Rational multi-section}]
    Let $Y\to^\pi X$ be a dominant map of $K$-varieties. A {\em rational multi-section} is a subvariety
    $U\subset Y$ such that the restriction ${\pi|_U}: U\to X$ is a generically finite dominant map.
\end{definition}

\begin{lemma}
\label{l:nonvac}
    Every dominant map $Y\to X$ admits a dense set of rational multi-sections, i.e. the closure of their union is all of $Y$.
\end{lemma}
\begin{proof}
    First assume that $X$ is irreducible. Let $\overline{K(X)}$ be an algebraic closure of the rational functions of $X$. Then every point of $Y(\overline{K(X)})$ is a germ of a rational multi-section, and, by Hilbert's Nullstellensatz, the closure of the union of all of these contains the generic fiber of $Y\to X$; in particular it is dense. For the general case, the argument above exhibits a dense set of rational multi-sections over each irreducible component. Their union gives a dense set of rational multi-sections of $Y\to X$.
\end{proof}

It will be useful to extend the definition of resolvent degree from generically finite dominant rational maps to all dominant rational maps.

\begin{definition}[{\bf Resolvent degree of a dominant map}]
\label{d:rddom}
    Let $Y\to^\pi X$ be a dominant map of $K$-varieties. The {\em resolvent degree of the dominant map}, $\RD(Y\to X)$ is defined to be the minimum $d$ for which there exists a dense set of rational multi-sections $\{U_\alpha\subset Y\}$  with $\RD(U_\alpha\to X)\le d$ for all $\alpha$.
\end{definition}

We will need a few basic facts about the resolvent degree of a dominant map.

\begin{lemma}
\label{lemma:first}
    Let $Y\to X$ be a dominant map of $K$-varieties.
    \begin{enumerate}
        \item $\RD(Y\to X)\le \dim(X)$.
        \item Let $Z\to X$ be any dominant map of $K$-varieties. Then
            \begin{equation*}
                \RD(Y\times_X Z\to Z)\le \RD(Y\to X).
            \end{equation*}
        \item If $Y\to X$ is birationally equivalent to $W\to Z$, then
            \begin{equation*}
                \RD(Y\to X)=\RD(W\to Z).
            \end{equation*}
        \item \label{l:firstmax} If $X=\bigcup X_i$ is a union of irreducible components, write $\{Y_{i,j}\}$ for the set of irreducible components of $Y$ which dominate $X_i$. Then
            \begin{equation*}
                \RD(Y\to X)=\max_{i,j}\{\RD(Y_{i,j}\to X_i)\}.
            \end{equation*}
    \end{enumerate}
\end{lemma}
\begin{proof}
    These follow immediately from the definition and the analogous properties for resolvent degree of generically finite dominant maps (cf. \cite[Lemmas 2.5, 2.6]{FW}).
\end{proof}

\begin{lemma}\label{l:surj}
    Let $Y\to X$ be a surjective map (on geometric points). Let $Z\to X$ be any map. Then
    \[
        \RD(Y|_Z\to Z)\le \dim(X).
    \]
\end{lemma}
\begin{proof}
    Let $W\subset X$ be the Zariski closure of the image of $Z\to X$. By construction, the map $Z\to W$ is dominant. The surjectivity of $Y\to X$ implies that the restriction
    \[
        Y|_W\to W
    \]
    is dominant. Therefore, by Lemma~\ref{lemma:first},
    \begin{align*}
        \RD(Y|_Z\to Z)&\le \RD(Y|_W\to W)\\
        &\le \dim(W)\\
        &\le\dim(X).
    \end{align*}
\end{proof}

\begin{lemma}\label{l:agree}
     Let $Y\to X$ be a generically finite dominant map. Then Definition \ref{d:rddom} specializes to Definition \ref{d:rdfdom} for $Y\to X$, i.e. they give equivalent notions of resolvent degree.
\end{lemma}
\begin{proof}
    By Lemma~\ref{lemma:first}~\ref{l:firstmax} and \cite[Lemma 2.6]{FW}, it suffices to prove this when $Y$ is irreducible. In this case, any rational multi-section $U\subset Y$ of $Y\to X$ must be dense in $Y$. In particular, it must be birational to $Y$. From the birational invariance of $\RD$ for generically finite dominant maps, we conclude that $\RD(U\to X)=\RD(Y\to X)$ (as generically finite dominant maps). The lemma follows.
\end{proof}

\begin{lemma}\label{lemma:max}
    Let $Z\to^{\pi_1} Y\to^{\pi_2} X$ be a pair of dominant maps of $K$-varieties. Then
    \begin{align*}
        \RD(Z\to X)&\ge \RD(Y\to X)\intertext{and}
        \RD(Z\to X)&\le\max\{\RD(Z\to Y),\RD(Y\to X)\}.
    \end{align*}
    with equality when either $Z\to Y$ or $Y\to X$ is generically finite.
\end{lemma}
\begin{proof}
    For the first inequality, let $\{U_\alpha\subset Z\}$ be a dense set of rational multi-sections of $Z\to X$ with $\RD(U_\alpha\to X)\le d$ for all $\alpha$. Then, shrinking each $U_\alpha$ as necessary (e.g. restricting to the preimage in $U$ of an affine open in $Y$), its (scheme theoretic) image $V_\alpha:=\Image(U_\alpha\to Y)$ is a subscheme of $Y$, and thus a rational multi-section of $Y\to X$. Since $Z\to Y$ is dominant, that $\{U_\alpha\subset Z\}$ is dense implies that $\{V_\alpha\subset Y\}$ is dense.  By \cite[Lemma 2.7]{FW}, we conclude that $\RD(U_\alpha\to X)\ge \RD(V_\alpha\to X)$. Minimizing over all $\{U_\alpha\subset Z\}$, we conclude that
    \[
        \RD(Z\to X)\ge \RD(Y\to X).
    \]
    For the second inequality, let $\{U_\alpha\subset Z\}$ be a dense set of rational multi-sections for $Z\to Y$ and $\{V_\beta\subset Y\}$ a dense set of rational multi-sections for $Y\to X$. Then
    \[
        \{W_{\alpha,\beta}:=U_\alpha\times_Y V_\beta\subset Z\}
    \]
    is a dense set of rational multi-sections for $Z\to X$. By \cite[Lemmas 2.5, 2.7]{FW},
    \[
        \RD(W_{\alpha,\beta}\to X)\le \max\{\RD(U_\alpha\to Y),\RD(V_\beta\to X)\}.
    \]
    Minimizing over all such collections $\{U_\alpha\},\{V_\beta\}$, we conclude
    \[
        \RD(Z\to X)\le \max\{\RD(Z\to Y),\RD(Y\to X)\}.
    \]
    To show the equalities when $\dim(Y)=\dim(X)$ or $\dim(Z)=\dim(Y)$, it suffices, by Lemma \ref{lemma:first}(\ref{l:firstmax}), to prove the case when $X$ and $Y$ are irreducible. Under this assumption, if $\dim(X)=\dim(Y)$ or if $\dim(Z)=\dim(Y)$, then any rational multi-section $U$ for $Z\to Y$ is a rational multi-section for $Z\to X$ and vice versa. In particular,
    \[
        \RD(U\to Y)\le \RD(U\to X)
    \]
    and taking the minimum over dense subsets of such, we see that $\RD(Z\to Y)\le \RD(Z\to X)$. The equality
    \[
        \RD(Z\to X)=\max\{\RD(Z\to Y),\RD(Y\to X)\}
    \]
    follows from what we have shown above.
\end{proof}

Special cases of the following are implicit in \cite{Seg1,Br1,Br}.
\begin{proposition}\label{p:degdfam}
    Let $Y\to X$ be a dominant map of $K$-varieties. Let $S\to X$ be a map such that the generic fiber is a Severi-Brauer variety over $K(X)$, and let $\overline{K(X)}$ be an algebraic closure of $K(X)$. Suppose that there exists an embedding over $X$
    \begin{equation*}
        Y\into S
    \end{equation*}
    such that the closure of the geometric generic fiber $Y|_{\overline{K(X)}}$ in $S|_{\overline{K(X)}}\cong\Pb^n_{\overline{K(X)}}$ has degree $d$. Then
    \begin{equation*}
        \RD(Y\to X)\le \RD(d)<d.
    \end{equation*}
\end{proposition}
\begin{proof}
    By the Merkurjev-Suslin theorem \cite[Theorem 16.1]{MS}, using that $K$ is a field of characteristic 0, there exists a solvable extension $L./K(X)$ such that $S|_{\Spec(L)}\cong \Pb^n_L$. Because we are in characteristic 0, the extension $L/K(X)$ is separable, so picking a primitive element $z$ and writing $L\cong K(X)(z)$, we can, by clearing denominators in the minimal polynomial for $z$ over $K(X)$ and using that the discriminant of this minimal polynomial is not identically 0, realize $L$ as $K(E)$ for $E\subset \Ab^1_X$ a locally closed subvariety such that the projection $E\to X$ is solvable and \'etale. Shrinking $E$ as needed, we can extend the above isomorphism $S|_{\Spec(L)}\cong\Pb^n_L$ to an isomorphism $S|_E\cong \Pb^n_E$. We conclude that the embedding $Y\into S$ pulls back to an embedding
    \[
        Y|_E\into S|_E\cong\Pb^n_E
    \]
    whose closure is a degree $d$ subvariety. Points of $Y|_E$ are thus of degree at most $d$ over $K(E)$ (and the generic point is of degree $d$). Therefore, by \cite[Lemma 2.9]{FW}, $Y|_E$ admits a dense set of rational multi-sections $\{U_\alpha\subset Y|_E\}$ with $\RD(U_\alpha\to E)\le \RD(d)$. The images of these rational multi-sections in $Y$, $\{V_\alpha\subset Y\}$ are thus a dense set of rational multi-sections, and by \cite[Lemma 2.6]{FW}, we have
    \begin{align*}
        \RD(V_\alpha\to X)&\le \RD(U_\alpha\to E\to X)\\
        &=\max\{\RD(U_\alpha\to E),\RD(E\to X)\}\\
        &\le \max\{\RD(d),1\}=\RD(d)<d.
    \end{align*}
\end{proof}

Now let $X$ be a variety, and let $\Ab^n_{X,\a}$ be the parameter space for $n$-valued algebraic functions on $X$ as in Section~\ref{s:tschirnt}. Observe that the action of $\Gb_m$ on algebraic functions by rescaling their values corresponds to a weighted action $\Gb_m\circlearrowleft \Ab^n_{X,\a}$ where
\[
    \lambda\cdot(a_1,\ldots,a_n)=(\lambda a_1,\ldots,\lambda^n a_n).
\]
Moreover, if $Z\subset\Ab^n_{X,\a}$ is weighted homogeneous with respect to this action, then $\ev^{-1}(Z)\subset\Tc_{\Phi}$ is homogeneous (with respect to the diagonal action of $\Gb_m$ on $\Ab^n_{X,\b}$).

\begin{lemma}\label{l:tschirn}
    Let $X$ be an irreducible $K$-variety. Let $\Phi$ be an algebraic function on $X$. Let $Z\subset\Ab^n_{X,\a}$ be a Zariski closed subvariety which is weighted homogeneous (relative to the above action). Let
    \begin{equation*}
        U\subset \ev^{-1}(Z)\subset\Tc_{\Phi}
    \end{equation*}
    be any rational multi-section for $\ev^{-1}(Z)\to X$. Then
    \[
        \RD(\Phi)\le \max\{\RD(U\to X),\dim(Z)-1\}.
    \]
\end{lemma}
\begin{proof}
    The multi-section $U\to \ev^{-1}(Z)$ determines a Tschirnhaus transformation $T$ of $\Phi|_U$ which is rational over $K(U)$. By the observations above, we have a pullback square
    \begin{equation*}
    \xymatrix{
        (E_\Phi)|_U \ar@{-->}[r] \ar[d] & (E_{P_n})|_Z \ar[d]\\
        U \ar@{-->}[r] & Z
    }
    \end{equation*}
    Since $Z$ is weighted homogeneous, we can projectivize $(E_{P_n})_Z\to Z$ to obtain a pullback square
    \begin{equation*}
    \xymatrix{
        (E_\Phi)|_U \ar@{-->}[r] \ar[d] & \Pb(E_{P_n})|_{\Pb(Z)} \ar[d]\\
        U \ar@{-->}[r] & \Pb(Z)
    }
    \end{equation*}
    where $\Pb(Z)\subset\Pb(\Ab^n_{\a})$ and $\Pb(\Ab^n_\a)$ now denotes the weighted projective space. The result now follows by applying Lemmas~\ref{lemma:first} and \ref{lemma:max}.
\end{proof}

\section{Hilbert's Formula for the Degree 9 and New General Upper Bounds}\label{s:beatbrauer}
We now apply the results of the previous sections to complete and extend Hilbert's argument from \cite{Hi}. We work throughout this section over an algebraically closed field $K$ of characteristic 0.

Let $\Hc_{d,N}$ denote the {\em parameter space} of degree $d$ hypersurfaces in $\Pb^N$, i.e. $\Hc_{d,N}\cong\Pb^{\binom{N+d}{d}-1}$. Let $\Mc_{d,N}$ denote the coarse moduli space of smooth hypersurfaces, i.e
\[
    \Mc_{d,N}=(\Hc_{d,N}-\Sigma)/\PGL_{N+1}
\]
where $\Sigma$ denotes the locus of singular hypersurfaces.  Let $\Hc_{d,N}^r$ denote the space of such hypersurfaces with a choice of $r$-plane on them, i.e. $\Hc_{d,N}^r$ is the incidence variety
\[
    \Hc_{d,N}^r:=\{(X,L)\in\Hc_{d,N}\times\Gr(r+1,N+1)~|~L\subset X\}.
\]
Similarly to above, let $\Mc_{d,N}^r$ denote the moduli of smooth degree $d$ hypersurfaces equipped with an incident $r$-plane, i.e.
\[
    \Mc_{d,N}^r=(\Hc_{d,N}^r-\tilde{\Sigma})/\PGL_{N+1},
\]
where $\tilde{\Sigma}\subset\Hc_{d,N}^r$ denotes the locus where the hypersurface is singular.

We will need the following theorem of Waldron \cite[Theorem 1.6]{Wa} (see also \cite[Theorem 1.2]{St}).
\begin{theorem}[Waldron]\label{t:rplanesurj}
    Let $d\ge 3$. The map
    \begin{equation*}
        \Hc_{d,N}^r\to\Hc_{d,N}
    \end{equation*}
    is surjective for $r$, $N$ such that
    \begin{equation*}
        (r+1)(N-r)-\binom{d+r}{r}\ge 0.
    \end{equation*}
\end{theorem}

Motivated by this theorem, we introduce the following notation:
\begin{notation}\label{n:beatbrauer}
    Given $(d,k)\in\Nb_{\ge 3}\times\Nb$, define
    \begin{align*}
        \psi(d,k)_0&=k.\intertext{For $0\le i<d-2$, define}
        \psi(d,k)_{i+1}&=\lceil\psi(d,k)_i+\binom{\psi(d,k)_i+d-i}{\psi(d,k)_i}/(\psi(d,k)_i+1)\rceil.\intertext{Finally, define}
        \psi(d,k)_{d-1}&=2\psi(d,k)_{d-2}+1.
    \end{align*}
\end{notation}
By Waldron's Theorem, for all $0\le i<d-2$, the map
\begin{equation*}
    \Hc_{d-i,\psi(d,k)_{i+1}}^{\psi(d,k)_i}\to\Hc_{d-i,\psi(d,k)_{i+1}}
\end{equation*}
is surjective. Similarly, by the classical theory of quadratic forms, the locus of smooth quadrics is contained in the image of the map
\begin{equation*}
    \Hc_{2,\psi(d,k)_{d-1}}^{\psi(d,k)_{d-2}}\to\Hc_{2,\psi(d,k)_{d-1}}
\end{equation*}
In words, the integers $\psi(d,k)_i$ are defined so that every smooth quadric in a $\Pb^{\psi(d,k)_{d-1}}$ contains a $\psi(d,k)_{d-2}$ plane, every cubic hypersurface in this $\psi(d,k)_{d-2}$ plane contains a $\psi(d,k)_{d-3}$ plane, every quartic in this $\psi(d,k)_{d-3}$ plane contains a $\psi(d,k)_{d-4}$ plane, and on down until we arrive at a $\psi(d,k)_1$ plane such that every degree $d$ hypersurface inside it contains a $k$-plane.
\begin{lemma}\label{l:dim}
    For all $d\ge 2$ and all $k\ge 1$,
    \begin{align*}
        \dim(\Mc_{3,\psi(d,k)_{d-2}})&\ge \max\{\dim(\Hc_{d-i,\psi(d,k)_{i+1}})\}_{i=0}^{d-3}\intertext{and}
        \dim(\Mc_{3,\psi(d,k)_{d-2}})+d+k+1&\ge\psi(d,k)_{d-1}+2.
    \end{align*}
\end{lemma}
\begin{proof}
    For each $i$,
    \begin{align*}
        \dim(\Hc_{d-i,\psi(d,k)_{i+1}})&=\binom{d-i+\psi(d,k)_{i+1}}{d-i}-1
    \end{align*}
    From the definition of the $\psi(d,k)_i$s, we conclude for all $i$ that
    \begin{align*}
        \dim(\Hc_{d-i,\psi(d,k)_{i+1}})\ge \dim(\Hc_{d-i+1,\psi(d,k)_i})
    \end{align*}
    and thus
    \begin{align*}
        \dim(\Hc_{4,\psi(d,k)_{d-3}})=\max\{\dim(\Hc_{d-i,\psi(d,k)_{i+1}})\}_{i=0}^{d-4}.
    \end{align*}
    Similarly,
    \begin{align*}
        \dim(\Mc_{3,\psi(d,k)_{d-2}})=
        \max\{0,\binom{3+\psi(d,k)_{d-2}}{3}-(\psi(d,k)_{d-2}+1)^2\}.
    \end{align*}
    From the definition, this is a maximum of a ceiling function of a monotone increasing degree 6 polynomial in $\psi(d,k)_{d-3}$, all of whose derivatives are monotone increasing in the domain $\psi(d,k)_{d-3}\ge 1$, while $\dim(\Hc_{4,\psi(d,k)_{d-3}})$ is a monotone increasing quartic, all of whose derivatives are monotone increasing in the same domain. Therefore, the inequality
    \begin{align*}
        \dim(\Mc_{3,\psi(d,k)_{d-2}})\ge \dim(\Hc_{4,\psi(d,k)_{d-3}})
    \end{align*}
    for all $(d,k)$ follows from the equality for $(d,k)=(3,1)$ and direct inspection of the higher derivatives of the sextic and quartic polynomials in the interval $\psi(d,k)_{d-3}\ge1$ 
    (for which both left and right hand side equal 4; note that the inequality is vacuously true for $(d,k)=(2,1)$).

    Finally, from the definition,
    \begin{align*}
        \psi(d,k)_{d-1}+2=2\psi(d,k)_{d-2}+3.
    \end{align*}
    By the same reasoning as above, the inequality
    \begin{equation*}
        \dim(\Mc_{3,\psi(d,k)_{d-2}})+d+k+1\ge \psi(d,k)_{d-1}+2
    \end{equation*}
    for all $(d,k)\in\Nb_{\ge 2}\times\Nb_{>0}$ follows from the inequality for $(d,k)=(2,1)$ (in which case the left hand side is 8 and the right hand side is 4).
\end{proof}

The lemma implies that for $d\ge 3$, $\dim(\Mc_{3,\psi(d,k)_{d-2}})$ gives a coarse upper bound on the resolvent degree of the surjective maps
\begin{align*}
    \Mc_{3,\psi(d,k)_{d-2}}^{\psi(d,k)_{d-3}}&\to\Mc_{3,\psi(d,k)_{d-2}}\\
    \Hc_{d-i,\psi(d,k)_{i+1}}^{\psi(d,k)_i}&\to\Hc_{d-i,\psi(d,k)_{i+1}}.
\end{align*}
This motivates the following definition.
\begin{definition}\label{d:FW}
    Given $(d,k)\in\Nb_{\ge 2}\times\Nb_{>0}$, define
    \begin{align*}
        \Phi(d,k):=\max\{\frac{(d+k)!}{d!}+1,\dim(\Mc_{3,\psi(d,k)_{d-2}})+d+k+1\}
    \end{align*}
    For $r\in \Nb_{\ge 4}$, define
    \begin{equation}\label{e:F}
        \FW(r):=2\lfloor \frac{1}{2}\cdot \left(\min_{d+k+1=r}\Phi(d,k)\right) \rfloor +1.
    \end{equation}
    For $r\le 3$, define $\FW(r)=r+1$.
\end{definition}

\begin{lemma}\label{l:inc}
    For all $r\in\Nb$, $\FW(r+1)>\FW(r)$, i.e. $\FW$ is monotone increasing.
\end{lemma}
\begin{proof}
    The maximum of two monotone increasing functions is monotone increasing, as is any linear combination with positive integer coefficients of the integer part of a monotone increasing function.
\end{proof}

We can now state our first main theorem.
\begin{theorem}\label{thm:beatbrauer}
    Let $\FW\colon \Nb\to\Nb$ be the monotone increasing function \eqref{e:F}. For all $n\ge \FW(r)$,
    \begin{equation*}
        \RD(n)\le n-r.
    \end{equation*}
\end{theorem}

\begin{example}
    Observe that
    \begin{align*}
        \FW(5)&=\Phi(3,1)=\max\{\frac{4!}{3!}+1,\dim(\Mc_{3,3})+5\}\\
            &=\max\{5,9\}=9.
    \end{align*}
    The theorem thus asserts that for $n\ge 9$, $\RD(n)\le n-5$, as first stated by Hilbert.
\end{example}

We can compare the upper bounds of Theorem~\ref{thm:beatbrauer} to Brauer's bounds as follows. Both the previous theorem and Brauer's theorem prove the existence, for each $r$, of an explicit cut-off (for $n$) after which $\RD(n)\le n-r$. More precisely, define
\[
    B(r):=(r-1)!+1.
\]
Brauer proved \cite[Theorem 1]{Br} that for $n\ge B(r)$,
\[
    \RD(n)\le n-r.
\]
The cut-off functions $B(r)$ and $\FW(r)$ are related as follows.
\begin{theorem}\label{t:compare}
    Let $B(r)$ and $\FW(r)$ be as above. There exists a monotone increasing function $\varphi\colon\Nb\to\Nb$, such that $\varphi(2)=5$, and such that for $r\ge \varphi(d)$,
    \begin{equation*}
    	B(r)/\FW(r)\ge d!
    \end{equation*}
    In particular, $\FW(r)\le B(r)$ for all $r$ and 
    \begin{equation*}
   		\lim_{r\rightarrow \infty} B(r)/\FW(r)=\infty.
    \end{equation*}
\end{theorem}

\begin{remark}\mbox{}
    \begin{enumerate}
    	\item As remarked above, Brauer's bound $B(r)$ gives the best prior general bound once $r\ge 7$; in this range, Theorem~\ref{t:compare} shows that $\FW$ is the best current bound. For $r=6$, Sylvester \cite{Sy} proved that the bound $n=44$ is sufficient, while for $r=5$, Segre and Dixmier proved that $n=9$ suffices.  In Appendix~\ref{a:bounds}, we give explicit computations of $\FW(r)$ for $r$ up to 15 (at which point $\FW(r)$ is approximately 3.6 billion). In particular, we see that $\FW(5)=9$ recovers the Hilbert-Wiman-Segre-Dixmier bound, and $\FW(6)=41$ improves Sylvester. 
        \item We do not expect that the upper bounds of Theorem~\ref{thm:beatbrauer} are themselves sharp for two reasons: first, we expect that further optimizations to the present method should be possible; and second, we have not made contact in this paper with the methods introduced by Sylvester and Hammond \cite{Sy,SH1,SH2} in their study of Hamilton's work \cite{Ham}.
    \end{enumerate}
\end{remark}

It remains to prove Theorems~\ref{thm:beatbrauer} and \ref{t:compare}.
\subsection{Proof of Theorem~\ref{thm:beatbrauer}}
Our proof follows the strategy outlined by Hilbert \cite{Hi}. We recall a classical lemma on quadrics.
\begin{lemma}\label{l:quads}
    Let $K$ be a field of characteristic 0, let $K\subset \overline{K}$ be an algebraic closure, and let $K^{2\text{-solv}}\subset \overline{K}$ denote the quadratic closure of $K$. For any smooth quadric $Q$ over $K$, with maximal isotropic Grassmannian $\Gr(Q)$, the inclusion
    \[
        \Gr(Q)(K^{2\text{-solv}})\subset\Gr(Q)(\overline{K})
    \]
    is Zariski dense. Moreover, for any $x\in \Gr(Q)(K^{2\text{-solv}})$, the associated Severi-Brauer variety over $K^{2\text{-solv}}$ is trivial.
\end{lemma}
\begin{proof}
    The proof is classical, and goes back at least to work of Sylvester. Recall that by completing the squares, every nonsingular, definite quadratic form $Q$ over $K$ admits a $K$-rational change of coordinates to one of the form
    \begin{equation}\label{eq:quad}
        Q'(x_1,\ldots,x_n)=a_1x_1^2+\cdots +a_n x_n^2
    \end{equation}
    for $a_i\in K^{\times}$. For example, see \cite{Fo} for explicit formulas for the $a_i$ in terms of minors of the matrix associated to the quadratic form (n.b. Fort states the results for real definite forms, but the method holds over any base field).

    Let $L=K(\sqrt{a_1},\ldots,\sqrt{a_n})\subset K^{2\text{-solv}}$. The $L$-rational change of coordinates
    \begin{equation*}
        x_i=:\frac{y_i}{\sqrt{a_i}}
    \end{equation*}
    converts the above quadratic form \eqref{eq:quad} to
    \begin{equation*}
        Q''(y_1,\ldots,y_n)=y_1^2+\cdots+y_n^2.
    \end{equation*}
    Finally, let $L'=L(\sqrt{-1})\subset K^{2\text{-solv}}$. Then the quadratic form $Q''$ vanishes identically on the linear subspace $\Lambda$ defined by
    \begin{align*}
        y_{2i-1}=\sqrt{-1}y_{2i}
    \end{align*}
    for $i=1,\ldots,\lfloor\frac{n}{2}\rfloor$. Counting the dimension, $\Lambda$ is a maximal isotropic, i.e.
    \begin{equation*}
        \Lambda\in \Gr(Q)(L')\subset\Gr(Q)(K^{2\text{-solv}}).
    \end{equation*}
    Using that $\Gr(Q)$ is a homogeneous space for the algebraic group $O(Q)$, and that $K$ (and thus $L'$) is an infinite field, we conclude that the $O(Q)(L')$ orbit of $\Lambda$ is dense in $\Gr(Q)(\bar{K})$ as claimed. Finally, because $\Lambda$ has an $L'$ point (e.g. for $n$ even $[y_1:\cdots:y_n]=[\sqrt{-1}:1:\cdots:\sqrt{-1}:1]$, with the analogous formula if $n$ is odd), the Severi-Brauer variety associated to $\Lambda$ over $L'$ splits completely. We conclude the same for every point in the $\O(Q)(L')$ orbit of $\Lambda$.
\end{proof}

\begin{corollary}\label{c:quads}
    Let $X$ be a variety over a field $K$ of characteristic 0. For any generically smooth family of quadrics $Q\to X$, the solvable multi-sections of $\Gr(Q)\to X$ are Zariski dense in $\Gr(Q)(\bar{K(X)})$.
\end{corollary}

\begin{proof}[Proof of Theorem~\ref{thm:beatbrauer}]
    Because $\FW$ is a monotone increasing function (by Lemma~\ref{l:inc}), if $n\ge \FW(r)$, then $n-1\ge \FW(r-1)$. We can therefore induct on $r$.

    For $n\le 4$, solutions in radicals imply $\RD(n)=1$.  That $\RD(n)\le n-4$ for $n\ge 5$ follows from Bring \cite{Bri} and Hamilton \cite{Ham}. We reprove this Bring-Hamilton bound as the base of our induction, in order to show the uniform general method; simple modifications of the below can be used to rederive the bound $\FW(r)$ for $r\le 3$.  
    
    For $n\ge 5$ we have a generically smooth family of quadrics $T_{12}\to \Ab^n_{\a}$ (by Lemma~\ref{l:quadsmooth}) of dimension at least 2. By Lemma~\ref{l:quads}, there exists a solvable branched cover
    \[
    	U_1\to \Ab^n_\a
    \]
    with a map over $\Ab^n_\a$ to the relative Grassmannian of maximal isotropics $\Gr(T_{12})$, i.e. there exists a linear embedding
    \[
    	\mathcal{L}\colon U_1\times\Pb^{\lfloor \frac{n-3}{2}\rfloor}\to T_{12}|_U. 
    \]
    Because $n\ge 5$, the dimension of the linear subspaces is at least 1. We can therefore intersect with $T_3|_{U_1}$ to get a rational map
    \begin{align*}
    	U_1&\dashrightarrow \Ab^3\\
    	u&\mapsto \mathcal{L}(u)\cap T_3.
    \end{align*}
    Adjoining the solution of this family of cubics, we get a solvable branched cover
    \[
    		U_2\to U_1
    \]
    and a map $U_2\to T_{123}$. By Lemma~\ref{l:tschirn}, we conclude that 
    \begin{align*}
	    	\RD(n)&\le \max\{\RD(U_2\to \Ab^n_{\a}),\dim(\Ab^{n-3}_{a_1=a_2=a_3=0})-1\}\\
	    	&=\max\{1,n-4\}=n-4.
    \end{align*}
    
    For the induction step, let $r\ge 5$ and assume that we have shown that for all $s<r$, $n\ge \FW(s)$ implies that $\RD(n)\le n-s$. Let $n\ge \FW(r)$. Note that if $\min_{d+k+1=r}\Phi(d,k)$ is odd, then the definition of $\FW$ implies that
    \begin{align*}
        \FW(r)&=\min_{d+k+1=r}\Phi(d,k).\intertext{Conversely, if $\min_{d+k+1=r}\Phi(d,k)$ is even, then}
        \FW(r)&=\min_{d+k+1=r}\Phi(d,k)+1.\intertext{Consequently, if $n$ is odd, then}
        n&\ge \min_{d+k+1=r}\Phi(d,k),\intertext{while if $n$ is even}
        n&\ge \min_{d+k+1=r}\Phi(d,k)+1.
    \end{align*}

    Let $(d,k)$ be such that
    \[
        \Phi(d,k)=\min_{d'+k'+1=r}\Phi(d',k').
    \]
    If $n$ is odd (and thus $n\ge \Phi(d,k)$), we will explicitly construct a rational multi-section
     \[
        U\to T_{1\cdots d+k}
     \]
     for $T_{1\cdots d+k}\to \Ab^n_{\a}$ with
     \[
        \RD(U\to\Ab^n_{\a})\le \max\{\RD(\frac{(d+k)!}{d!}),\dim(\Mc_{3,\psi(d,k)_{d-2}})\}.
     \]
     If $n$ is even (and thus $n\ge \Phi(d,k)+1$), {\em mutatis mutandis} the same argument will produce a rational multi-section
     \[
        U\to T'_{1\cdots d+k}
     \]
     with $\RD(U\to\Ab^n_{\a})\le \max\{\RD(\frac{(d+k)!}{d!}),\dim(\Mc_{3,\psi(d,k)_{d-2}})\}$.
     \medskip
     \newline
     \noindent
     {\bf Case 1: $n$ odd.} Let $U_1=\Ab^n_{\a}$. By Lemma~\ref{l:quadsmooth}, the family $T_{12}\to\Ab^n_{\a}$ is generically smooth. By Corollary~\ref{c:Hfix}, there exists a dense open $V\subset \Gr(T_{12})$, such that
     \[
        \Lc|_V\times_{\Pb^n_{\Ab^n_{\a}}} T_{123}\to \Ab^n_{\a}
     \]
     is smooth (i.e. for the generic polynomial, the intersection of $T_{123}(\a)$ with a generic maximal isotropic in $T_{12}(\a)$ is smooth).

     By Corollary~\ref{c:quads},
     \[
        \RD(V\to\Ab^n_{\a})=1
     \]
     More precisely, there exists a multi-section $U_2\subset V$ such that $U_2\to U_1$ is a solvable cover of its image, and such that
     \[
        \Lc|_{U_2}\cong\Pb^{\frac{n-3}{2}}_{U_2}.
     \]
     Now, by Lemma~\ref{l:dim} and our assumption on $n$,
     \begin{align*}
        n&\ge \Phi(d,k)\ge \psi(d,k)_{d-1}+2\\
        &=2\psi(d,k)_{d-2}+3.\intertext{Therefore,}
        \frac{n-3}{2}&\ge \psi(d,k)_{d-2}
     \end{align*}
     If $\frac{n-3}{2}=\psi(d,k)_{d-2}$, then we obtain a map
     \begin{align*}
        U_2&\to\Mc_{3,\psi(d,k)_{d-2}}\\
        x&\mapsto \Lc|_x\times_{\Pb^{n-2}}T_{123}|_x.
     \end{align*}
     If $\frac{n-3}{2}>\psi(d,k)_{d-2}$, by the Bertini Theorem for isotropics (Proposition~\ref{p:bertini}), there exists a dense open
     \[
        V'\subset\Gr(\psi(d,k)_{d-2},\Lc|_{U_2})
     \]
     such that the family of cubic hypersurfaces in $\Pb^{\psi(d,k)_{d-2}}$ given by
     \[
        V'\times_{\Lc|_{U_2}}(T_{123}\times_{\Pb^{n-2}_{U_1}}\Lc|_{U_2})\to U_2
     \]
     is generically smooth. Because rational points are dense in Grassmannians, perhaps after shrinking $U_2$, we obtain a section $U_2\to V'$. As above, we again obtain a map
     \begin{align*}
        U_2\to^{\cap T_{123}} \Mc_{3,\psi(d,k)_{d-2}}.
     \end{align*}
    Note that, from the construction above, $\RD(U_2\to U_1)=1$.

    By Waldron's Theorem (Theorem~\ref{t:rplanesurj}) and the definition of the numbers $\psi(d,k)_i$, the map
    \begin{equation*}
        \Mc_{3,\psi(d,k)_{d-2}}^{\psi(d,k)_{d-3}}\to\Mc_{3,\psi(d,k)_{d-2}}
    \end{equation*}
    is surjective. Therefore, the map
    \[
        \Mc_{3,\psi(d,k)_{d-2}}^{\psi(d,k)_{d-3}}|_{U_2}\to U_2
    \]
    is surjective, and by Lemma~\ref{l:surj},
    \[
        \RD(\Mc_{3,\psi(d,k)_{d-2}}^{\psi(d,k)_{d-3}}|_{U_2}\to U_2)\le \dim(\Mc_{3,\psi(d,k)_{d-2}}).
    \]
    Let $U'\subset\Mc_{3,\psi(d,k)_{d-2}}^{\psi(d,k)_{d-3}}|_{U_2}$ be any rational multi-section such that
    \[
        \RD(U'\to U_2)=\RD(\Mc_{3,\psi(d,k)_{d-2}}^{\psi(d,k)_{d-3}}|_{U_2}\to U_2).
    \]
    Let $\bar{\Lc}\to \Mc_{3,\psi(d,k)_{d-2}}^{\psi(d,k)_{d-3}}$ denote the tautological $\psi(d,k)_{d-3}$-plane bundle. By the Merkurjev-Suslin Theorem \cite[Theorem 16.1]{MS}, there exists a solvable \'etale map $U_3\to U'$ such that
    \[
        \bar{\Lc}|_{U_3}\cong \Pb^{\psi(d,k)_{d-3}}_{U_3}.
    \]
    By Lemma~\ref{lemma:max} and the construction above,
    \[
        \RD(U_3\to U_2)=\max\{\RD(U'\to U_2),1\}\le \dim(\Mc_{3,\psi(d,k)_{d-2}}).
    \]
    Further, intersecting with the Tschirnhaus hypersurface $T_4$, we obtain a map
    \begin{align*}
        U_3&\to^{\cap T_4}\Hc_{4,\psi(d,k)_{d-3}}\\
        x&\mapsto (T_{123}|_x\times_{U_3}\bar{\Lc}|_{U_3})\times_{\Pb^{n-1}_{U_3}} T_4|_{U_3}.
    \end{align*}

    By induction, we now construct, for each $4\le i\le d$, a quasi-finite dominant map
    \[
        U_i\to U_{i-1}
    \]
    such that
    \begin{enumerate}
        \item $\RD(U_i\to U_{i-1})\le \dim(\Hc_{i,\psi(d,k)_{d-i+1}})$,
        \item we have a commuting diagram
                \begin{equation*}
                    \xymatrix{
                        U_i \ar[rr] \ar[d] && \Hc_{i,\psi(d,k)_{d-i+1}}^{\psi(d,k)_{d-i}}\ar[d] \\
                        U_{i-1} \ar[rr]^{\cap T_i} && \Hc_{i,\psi(d,k)_{d-i+1}}
                    }
                \end{equation*}
                with a trivialization
                \[
                    \Lc|_{U_i}\cong \Pb^{\psi(d,k)_{d-i}}_{U_i},
                \]
                where $\Lc\to\Hc_{i,\psi(d,k)_{d-i+1}}^{\psi(d,k)_{d-i}}$ denotes the tautological $\psi(d,k)_{d-i}$-plane bundle;
        \item and the assignment
            \[
                x\mapsto (T_{1\cdots i}|_x\times_{U_i} \Lc_{i,\psi(d,k)_{d-i+1}}|_{U_i})\times_{\Pb^{n-1}_{U_i}} T_{i+1}|_x
            \]
            defines a map
            \[
                U_i\to^{\cap T_{i+1}} \Hc_{i+1,\psi(d,k)_{d-i}}.
            \]
    \end{enumerate}
    The construction proceeds along the same lines as the construction of $U_3$ above. Given $U_{i-1}$ with the map
    \[
        U_{i-1}\to^{\cap T_i}\Hc_{i,\psi(d,k)_{d-i+1}},
    \]
    by the definition of the $\psi(d,k)_j$s and Waldron's Theorem (Theorem~\ref{t:rplanesurj}), the map
    \begin{equation*}
        \Hc_{i,\psi(d,k)_{d-i+1}}^{\psi(d,k)_{d-i}}\to\Hc_{i,\psi(d,k)_{d-i+1}}
    \end{equation*}
    is surjective. Therefore, the map
    \[
         \Hc_{i,\psi(d,k)_{d-i+1}}^{\psi(d,k)_{d-i}}|_{U_{i-1}}\to U_{i-1}
    \]
    is surjective, and by Lemma~\ref{l:surj},
    \[
        \RD(\Hc_{i,\psi(d,k)_{d-i+1}}^{\psi(d,k)_{d-i}}|_{U_{i-1}}\to U_{i-1})\le \dim(\Hc_{i,\psi(d,k)_{d-i+1}}).
    \]
    Let $U'\subset\Hc_{i,\psi(d,k)_{d-i+1}}^{\psi(d,k)_{d-i}}|_{U_{i-1}}$ be any rational multi-section such that
    \[
        \RD(U'\to U_{i-1})=\RD(\Hc_{i,\psi(d,k)_{d-i+1}}^{\psi(d,k)_{d-i}}|_{U_{i-1}}\to U_{i-1}).
    \]
    Let $\Lc\to \Hc_{i,\psi(d,k)_{d-i+1}}^{\psi(d,k)_{d-i}}$ denote the tautological $\psi(d,k)_{d-i}$-plane bundle. By the Merkurjev-Suslin Theorem \cite[Theorem 16.1]{MS}, there exists a solvable \'etale map $U_i\to U'$ such that
    \[
        \Lc|_{U_i}\cong \Pb^{\psi(d,k)_{d-i}}_{U_i}.
    \]
    By Lemma~\ref{lemma:max} and the construction above,
    \[
        \RD(U_i\to U_{i-1})=\max\{\RD(U'\to U_{i-1}),1\}\le \dim(\Hc_{i,\psi(d,k)_{d-i+1}}).
    \]
    Finally, to complete the induction step, we observe that, by intersecting with the Tschirnhaus hypersurface $T_{i+1}$, we obtain a map
    \begin{align*}
        U_i&\to^{\cap T_{i+1}}\Hc_{i+1,\psi(d,k)_{d-i}}\\
        x&\mapsto (T_{1\cdots i}|_x\times_{U_i}\Lc|_{U_i})\times_{\Pb^{n-1}_{U_i}} T_{i+1}|_{U_i}.
    \end{align*}
    This completes the induction step. We have thus constructed a tower of maps
    \[
        U_d\to\cdots \to U_4\to U_3\to U_2\to U_1=\Ab^n_{\a}.
    \]
    Further, from the inductive construction and Lemmas~\ref{lemma:max} and \ref{l:dim}, we have
    \[
        \RD(U_d\to \Ab^n_{\a})\le\dim(\Mc_{3,\psi(d,k)_{d-2}}).
    \]
    Now let $\Lc\to\Hc_{d,\psi(d,k)_1}$ denote the tautological $k$-plane bundle (n.b. $k=\psi(d,k)_0$). Then, by construction, we have an isomorphism
    \[
        \Lc|_{U_d}\cong \Pb^k_{U_d}.
    \]
    For $i_1<\ldots<i_k$, and $N$, let
    \[
        \Hc_{i_1\cdots i_k,N}
    \]
    denote the parameter space of complete intersections of degree $(i_1,\ldots,i_k)$. Let
    \[
        \Ic\to  \Hc_{i_1\cdots i_k,N}
    \]
    denote the tautological family of complete intersections. By Proposition~\ref{p:degdfam},
    \[
        \RD(\Ic\to\Hc_{i_1\cdots i_k,N})\le \RD(i_1\cdots i_k).
    \]
    By our inductive construction, we have a map
    \begin{align*}
        U_d&\to^{\cap T_{(d+1)\cdots(d+k)}}\Hc_{(d+1)\cdots(d+k),k}\\
        x&\mapsto (T_{1\cdots d}|_x\times_{U_d}\Lc|_{U_d})\times_{\Pb^{n-1}_{U_d}} T_{(d+1)\cdots(d+k)}|_{U_d}.
    \end{align*}
    Because, $\Ic\to\Hc_{(d+1)\cdots(d+k),k}$ is surjective, by Lemma~\ref{l:surj},
    \[
        \RD(\Ic|_{U_d}\to U_d)\le\RD(\frac{(d+k)!}{d!}).
    \]
    Let $U_{d+1}\subset\Ic|_{U_d}$ be a rational multi-section of $\Ic|_{U_d}\to U_d$ such that
    \[
        \RD(U_{d+1}\to U_d)\le\RD(\frac{(d+k)!}{d!}).
    \]
    Then, by construction, $U_{d+1}$ carries a canonical map
    \[
        U_{d+1}\to T_{1\cdots(d+k)}
    \]
    making it a rational multi-section of the Tschirnhaus complete intersection. Further, by the above construction and Lemma~\ref{lemma:max},
    \begin{align*}
        \RD(U_{d+1}\to \Ab^n_{\a})&\le\max\{\RD(\frac{(d+k)!}{d!}),\dim(\Mc_{3,\psi(d,k)_{d-2}})\}.
    \end{align*}
    By assumption, $n\ge \FW(r)=\Phi(d,k)\ge \frac{(d+k)!}{d!}+1$. Lemma~\ref{l:inc} thus implies that $\frac{(d+k)!}{d!}\ge \FW(r-1)$.  Therefore, by the inductive hypothesis,
    \[
        \RD(\frac{(d+k)!}{d!})\le \frac{(d+k)!}{d!}-(r-1).
    \]
    Moreover, from the definition of $\Phi(d,k)$, $n\ge\Phi(d,k)$ implies that $n\ge \dim(\Mc_{3,\psi(d,k)_{d-2}})+r$.

    By Lemma~\ref{l:tschirn}, we therefore conclude that
    \begin{align*}
        \RD(n)&\le\max\{\RD(U_{d+1}\to\Ab^n_{\a}),\dim(\ev(T_{1\cdots(d+k)}))-1\}\\
        &\le\max\{\frac{(d+k)!}{d!}-(r-1),\dim(\Mc_{3,\psi(d,k)_{d-2}}),n-r.\}\\
        &=n-r.
    \end{align*}
     \medskip
     \noindent
     {\bf Case 2: $n$ even.} Let $U_1=\Ab^n_{\a}$. By Lemma~\ref{l:quadsmooth}, the family $T'_{12}\to\Ab^n_{\a}$ is generically smooth. By Corollary~\ref{c:Hfix}, there exists a dense open $V\subset \Gr(T'_{12})$, such that
     \[
        \Lc|_V\times_{\Pb^n_{\Ab^n_{\a}}} T'_{123}\to \Ab^n_{\a}
     \]
     is smooth (i.e. for the generic polynomial, the intersection of $T'_{123}(\a)$ with a generic maximal isotropic in $T'_{12}(\a)$ is smooth).

     By Corollary~\ref{c:quads},
     \[
        \RD(V\to\Ab^n_{\a})=1
     \]
     More precisely, there exists a multi-section $U_2\subset V$ such that $U_2\to U_1$ is a solvable cover of its image, and such that
     \[
        \Lc|_{U_2}\cong\Pb^{\frac{n}{2}-2}_{U_2}.
     \]
     Now, by Lemma~\ref{l:dim} and our assumption on $n$
     \begin{align*}
        n-1&\ge \Phi(d,k)\ge \psi(d,k)_{d-1}+2\\
        &=2\psi(d,k)_{d-2}+3.\intertext{Therefore,}
        \frac{n}{2}-2&\ge \psi(d,k)_{d-2}
     \end{align*}
     If $\frac{n}{2}-2=\psi(d,k)_{d-2}$, then we obtain a map
     \begin{align*}
        U_2&\to\Mc_{3,\psi(d,k)_{d-2}}\\
        x&\mapsto \Lc|_x\times_{\Pb^{n-2}}T'_{123}|_x.
     \end{align*}
     If $\frac{n}{2}-2>\psi(d,k)_{d-2}$, by the Bertini Theorem for isotropics (Proposition~\ref{p:bertini}), there exists a dense open
     \[
        V'\subset\Gr(\psi(d,k)_{d-2},\Lc|_{U_2})
     \]
     such that the family of cubic hypersurfaces in $\Pb^{\psi(d,k)_{d-2}}$ given by
     \[
        V'\times_{\Lc|_{U_2}}(T'_{123}\times_{\Pb^{n-2}_{U_1}}\Lc|_{U_2})\to U_2
     \]
     is generically smooth. Because rational points are dense in Grassmannians, perhaps after shrinking $U_2$, we obtain a section $U_2\to V'$. As above, we again obtain a map
     \begin{align*}
        U_2\to^{\cap T'_{123}} \Mc_{3,\psi(d,k)_{d-2}}.
     \end{align*}
    Note that, from the construction above, $\RD(U_2\to U_1)=1$. The remainder of the proof now proceeds exactly as in the case of $n$ odd.
\end{proof}

\subsection{Proof of Theorem~\ref{t:compare}}

\begin{proof}[Proof of Theorem~\ref{t:compare}]
    We deduce the theorem from the following:
    \begin{claim}\label{claim:comp}
        There exists a monotone increasing function $\rho\colon\Nb\to\Nb$ such that
        \begin{enumerate}
            \item for $k\ge \rho(d)$,
                    \begin{align*}
                        \frac{(d+k)!}{d!}+1&=\Phi(d,k)\\
                            &\le \Phi(d-1,k+1)
                        \end{align*}
                    (i.e. both conditions hold for $k\ge \rho(d)$);
            \item for all $k<\rho(d)$, either 
                    \[
                        \Phi(d,k)>\Phi(d-1,k+1).
                    \]
                    or 
                    \[
                    	 \frac{(d+k)!}{d!}+1\neq\Phi(d,k)
                    \]
                    (i.e. $\rho(d)$ is the least integer such that both conditions hold).
        \end{enumerate}
    \end{claim}
    Granting the claim, let $\varphi(d):=\rho(d+1)+d+2$. From Definition~\ref{d:FW}, we see that $\rho(3)=2$, and thus $\varphi(2)=6$. However, $\FW(5)=9$ while $B(5)=25$, so we can modify $\varphi$ by setting $\varphi(2):=5$ as claimed. Moreover, for $r\ge\varphi(d)$, we have
    \begin{align*}
        k&:=(r-1)-(d+1)\\
        &\ge \varphi(d)-(d+2)\\
        &\ge \rho(d+1)\intertext{As a result,}
        \FW(r)&=\frac{(r-1)!}{(d+1)!}+1\intertext{and therefore,}
        B(r)/\FW(r)&=\frac{(r-1)!+1}{(r-1)!/(d+1)!+1}\\
        &\ge d!
    \end{align*}
    We now prove Claim~\ref{claim:comp} by asymptotic estimates; more precisely, we show that for each $d$, $\dim(\Mc_{3,\psi(d,k)_{d-2}})$ grows polynomially in $k$, while $\frac{(d+k)!}{d!}$ grows superexponentially. Precise formulas for the function $\rho$ require a more detailed analysis.

    Continuing to follow Notation~\ref{n:beatbrauer}, we claim the following:
    \begin{claim}\label{claim:beatbrauer}
        Fix $d$.  Then as a function of $k$,
        \begin{equation*}
            \Oc((d+k)!)\ge\max\{\Oc(\dim(\Mc_{3,\psi(d,k)_{d-2}})),\Oc(\dim(\Mc_{3,\psi(d-1,k+1)_{d-3}}))\}),
        \end{equation*}
        where $\Oc(f)$ denotes the asymptotic growth of a function $f$.
    \end{claim}
    Granting the claim, we see that for $k>>d$,
    \begin{equation*}
        \Phi(d,k)=\frac{(d+k)!}{d!}<\frac{(d+k)!}{(d-1)!}=\Phi(d-1,k+1).
    \end{equation*}
    Note that by definition,
    \begin{align*}
        \Phi(d,k)&=\max\{\frac{(d+k)!}{d!}+1,\dim(\Mc_{3,\psi(d,k)_{d-2}})+d+k+1\}
    \end{align*}
    Therefore Claim~\ref{claim:comp} follows from Claim~\ref{claim:beatbrauer}. To prove Claim~\ref{claim:beatbrauer}, recall Stirling's formula (cf. \cite{Ro})
    \begin{equation*}
        \sqrt{2\pi}m^{m+\frac{1}{2}}e^{\frac{1}{12m+1}-m}\le m!\le \sqrt{2\pi}m^{m+\frac{1}{2}}e^{\frac{1}{12m}-m}
    \end{equation*}
    This implies that
    \begin{equation*}
        \Oc(\ln((d+k)!))=\Oc((d+k+\frac{1}{2})\ln(d+k)).
    \end{equation*}
    It suffices to prove that
    \begin{equation*}
        \max\{\Oc(\dim(\Mc_{3,\psi(d,k)_{d-2}})),\Oc(\dim(\Mc_{3,\psi(d-1,k+1)_{d-3}}))\}=\Oc(k^{\alpha_d})
    \end{equation*}
    for some $\alpha_d$, as then
    \begin{align*}
        \max\{\Oc(\ln(\dim(\Mc_{3,\psi(d,k)_{d-2}}))),\Oc(\ln(\dim(\Mc_{3,\psi(d-1,k+1)_{d-3}})))\}&=\Oc(\alpha_d\cdot \ln(k))\\
        &\le \Oc((d+k+\frac{1}{2})\ln(d+k))\\
        &=\Oc(\ln(d+k)!).
    \end{align*}

    Recall that $\psi(d,k)_0=k$ and for $i>0$,
    \begin{align*}
        \psi(d,k)_i&=\lceil\psi(d,k)_{i-1}+\binom{\psi(d,k)_{i-1}+d-(i-1)}{\psi(d,k)_{i-1}}/(\psi(d,k)_{i-1}+1)\rceil.\intertext{Therefore}
        \psi(d,k)_i&\sim \frac{(d-i+1+\psi(d,k)_{i-1})\cdots(\psi(d,k)_{i-1}+2)}{(d-i+1)!}\sim (\psi(d,k)_{i-1})^{d-i}.\intertext{Because $\psi(d,k)_1\sim k^{d-1}$, by induction, we obtain}
        \dim(\Hc_{d-i,\psi(d,k)_{i+1}})&=\binom{d-i+\psi(d,k)_{i+1}}{\psi(d,k)_{i+1}}-1\\
        &\sim \psi(d,k)_{i+1}^{d-i}\\
        &\sim k^{(d-i)\frac{(d-1)!}{(d-i-2)!}}.\intertext{Similarly,}
        \dim(\Mc_{3,\psi(d,k)_{d-2}})&\sim k^{3(d-1)!}.\intertext{By the same argument,}
        \dim(\Mc_{3,\psi(d-1,k+1)_{d-3}})&\sim (k+1)^{3(d-2)!}\sim k^{3(d-2)!},\intertext{and, thus, as functions of $k$,}
        \Oc((d+k)!)&\ge\max\{\Oc(\dim(\Mc_{3,\psi(d,k)_{d-2}})),\Oc(\dim(\Mc_{3,\psi(d-1,k+1)_{d-3}}))\}
    \end{align*}
    as claimed.
\end{proof}

\newpage

\appendix
\section{Explicit Bounds}\label{a:bounds}
\begin{table}[ht]
\caption{Upper Bounds on $\RD(n)$}
\centering
\begin{tabular}{c c c c c c c}
\hline\hline
$r$ & $\FW(r)$ & Best Prior Bound $B'(r)$ & Source of $B'(r)$  & $B'(r)/\FW(r)$ & $(d,k)$ \\[0.5ex]
\hline
2 &     3 &          3          & Babylonians           & 1          &        \\
3 &     4 &          4          & Ferrari               & 1          &        \\
4 &     5 &          5          & Bring \cite{Bri}      & 1          & (2,1)    \\
5 &     9 &          9          & Segre \cite{Seg1}       & 1          & (3,1)  \\
6 &     41 &         44         & Sylvester \cite{Sy}   & 1.07          & (3,2) \\
7 &     121 &        721        & Brauer \cite{Br}      & 5.95      & (3,3) \\
8 &     841 &        5041       & ''                    & 5.99       & (3,4)\\
9 &     6721 &       40321      & ''                    & 5.99      & (3,5)  \\
10 &    60481 &      362881     & ''                    & 5.99      & (3,6)   \\
11 &    604801 &     3628801    & ''                    &   5.99      & (3,7) \\
12 &    6652801 &    39916801   & ''                    &  5.99      & (3,8)  \\
13 &    78485043 &   $12!+1$    & ''                    &   6.10 &      (4,8) \\
14 &    320082459 &  $13!$+1    & ''                    &   19.45 &     (4,9) \\
15 &    3632428801 & $14!$+1    & ''                    &   24 &        (4,10)\\ [1ex]
\hline
\end{tabular}\label{table:bounds}
\caption{In the rightmost column above, $k$ is the dimension of the linear subspace on the degree $d$ hypersurface that we use to construct the necessary Tschirnhaus transformation, e.g. for $r=5$,  $(d,k)=(3,1)$ and we are using a line on a cubic surface  {\em \`{a} la} Hilbert to prove $\FW(5)=9$.}
\end{table}

\section{Historical Background.}\label{a:hist}
	\begin{quote}
		``The theory has been a plant of slow growth.''
		\attrib{Sylvester and Hammond, 1887\footnote{\cite[p. 286]{SH1}}}
	\end{quote}
	Tschirnhaus \cite{Ts} introduced his transformation
	to show that $\RD(n)\le n-3$, improving upon the linear change of variables used by the Babylonians to set the first coefficient of the general polynomial to 0. A century later, Bring \cite{Bri} improved this for $n=5$ to show that $\RD(5)=1$. Hamilton \cite{Ham} was the first to show that
	\begin{equation*}
	\lim_{n\to\infty} n-\RD(n)=\infty.
	\end{equation*}
	More precisely, he showed the existence a monotone increasing function $H\colon\Nb\to\Nb$, such that $n-\RD(n)\ge r$ for $n\ge H(r)$.\footnote{The numbers $H(r)$ are listed as the ``Hamilton numbers'' in the {\em Online Encyclopedia of Integer Sequences}.} Hamilton computed the initial values of $H$ (for $r\le 7$). Five decades later, Sylvester \cite{Sy} extended Hamilton's computations to give:
	\begin{equation*}
	\begin{array}{c|c|c|c|c|c|c}
	r & 4 & 5 & 6 & 7 & 8 & 9 \\ \hline
	H(r) & 5 & 11 & 47 & 923 & 409,619 & 83,763,206,255
	\end{array}
	\end{equation*}
	Sylvester then sharpened Hamilton's bounds slightly (see \cite[p. 485]{Sy})\footnote{Writing $S(r)$ for Sylvester's sharpening, the initial values are $S(4)=5$, $S(5)=10$, $S(6)=44$, $S(7)=905$.}, and Sylvester and Hammond \cite{SH1}, \cite{SH2} gave a generating function for $H$.
	
	Preceding Sylvester (and apparently unbeknownst to him at the time of \cite{SH1}), Klein \cite{KleinFirst} initiated a new approach to solving polynomials, linking it with group theory, representation theory, projective geometry, classical invariant theory, and the theory of elliptic and automorphic functions. Fundamental to Klein's vision was the goal of reducing a given algebraic function to a simplest possible ``normal form'', with the ideal being a normal form given by the action of the monodromy group of the function on a projective space of minimal dimension.\footnote{As Wiman proved, this program cannot produce a solution in $\RD(n)$ variables for the general degree $n$ polynomial once $n$ is at least $8$.} For $n=5,6,7$, this program allowed Klein \cite{KleinIcos,Klein67,KleinLast} to reproduce the Bring/Hamilton bounds of $\RD(n)\le n-4$ with substantial simplifications in both the algebra of the formulas and the geometry of the normal forms involved. Klein also popularized the problem of finding simplest solutions of polynomials \cite[Second Part, Ch. II]{KleinEl}, was the first, or among the first, to explicitly consider the problem of lower bounds for $\RD$ \cite{KleinEv,KleinLast}, and worked, over a 50 year span, to anchor this problem firmly within the central mathematical concerns of his time (see also \cite{KleinLetter,Klein8}, and more generally \cite{KleinCW,Fr}).
	
	In his 1900 address at the Universal Exposition in Paris, Hilbert \cite[Problem 13]{Hi1} explicitly posed the problem of the non-existence of 2-variable formulas for the general degree 7 polynomial. Hilbert's address cements two decisive shifts for the problem: first, he explicitly called attention to the question of lower bounds on resolvent degree, made conjectures as to lower bounds, and advocated for this as the fundamental problem. Second, Hilbert built upon Enriques' 1897 ICM address \cite{En} by generalizing the problem to encompass formulas using analytic functions and even continuous ones; he then proved by a dimension count that the general three variable analytic function does not admit a formula in analytic functions of two or fewer variables. Hilbert returned to this problem at the end of his career in \cite{Hi}, where he explicitly conjectured that $\RD(6)=2$, $\RD(7)=3$, $\RD(8)=4$, and then sketched a beautiful geometric idea to lower $\RD(9)$ to at most 4. Shortly after, Wiman \cite{Wi} sketched another approach to showing $\RD(n)\le n-5$ for $n\ge 9$. As Dixmier observed \cite{Di}, there are gaps in both Hilbert and Wiman's proofs due to their assuming certain forms are sufficiently generic.
	
	Progress on the general problem of bounding $\RD(n)$ stalled after Hilbert. N. Chebotarev highlighted this and related questions in his 1932 ICM address \cite{TschebICM}, and in several papers in the 1930s and 1940s \cite{Tscheb1,Tscheb2,Tscheb3,Tscheb4}. However, by the mid-20th century, much of the 19th century work appears to have been forgotten. Segre \cite{Seg1}, building on Hilbert, provided the first rigorous proof that $\RD(n)\le n-5$ for $n\ge 9$, and proved that for $n\ge 157$, $\RD(n)\le n-6$ (n.b. Hamilton proved this for $n\ge 47$, while Sylvester proved it for $n\ge 44$). G. Chebotarev (N.'s son) worked to extend Wiman's methods to show $\RD(n)\le n-6$ for $n\ge 21$ \cite{Ch}, but his proof is incomplete.\footnote{As remarked above, Chebotarev's argument has the same gap that Dixmier \cite{Di} observed in Hilbert and Wiman, namely certain non-generic forms are assumed to be generic.} Segre ({\em loc. cit.}) conjectured that in the limit
	\begin{equation*}
	\lim_{n\to\infty} n-\RD(n)=\infty.
	\end{equation*}
	(i.e. precisely what Hamilton had showed over a century earlier).  Brauer \cite{Br1} and Segre each reproved this statement, but without giving effective bounds {\em \`a la} Hamilton  (see also \cite{Seg2}).
	
	In 1957, Arnold (then 19 years old) published a theorem which he described as a ``complete solution of the 13th problem of Hilbert'' \cite{Ar57}. A strengthening of Arnold's theorem, published soon after by Kolmogorov \cite{Ko}, states that for any continuous map $f\colon [0,1]^n\to \Rb$, there exist continuous functions $g_j,\varphi_{ij}\colon[0,1]\to\Rb$ such that
	\[
	f(t_1,\ldots,t_n)=\sum_{j=1}^{2n-1}g_j(\sum_{i=1}^k\varphi_{ij}(t_i))
	\]
	To apply this to Hilbert's problem, one must interpret Hilbert as having asked for an obstruction to expressing a single-valued branch of the general degree 7 polynomial as a composition of (single-valued) continuous functions of two or fewer variables. Following Arnold and Kolmogorov, work on the problem in all of its forms largely collapsed, this despite Arnold's efforts over a four decade span \cite{Ar70a,Ar70b,Ar70c,AS,ArLast} to call attention to and solve Hilbert's (still open!) thirteenth problem.\footnote{See also \cite[Problems  1972-27, 1976-34, 1979-10, 1980-10, 1985-18]{ArProb}}
	
	In 1971, Khovanskii \cite{Kh} showed that if one prohibited the use of division in a formula (i.e. one only allowed ``entire'' algebraic functions), then the quintic was not solvable in 1-variable functions.\footnote{A late paper of Abhyankar \cite{Ab}, apparently unaware of Khovanskii's result, proves the analogous theorem for the sextic.} Khovanskii emphasized that, more than anything else, this result shows the importance of division.\footnote{Lin has also extensively investigated what one can say for the general degree $n$ polynomial if one rules out division and possibly imposes further restrictions, see the papers \cite{Li1,Li2,Li3}.}
	
	In 1975, Brauer \cite{Br} gave the first rigorous definition of resolvent degree in the literature (followed soon after by Arnold and Shimura \cite{AS}). Brauer then proved that for $n\ge (r-1)!+1$, $\RD(n)\le n-r$.  This improves Sylvester and Hamilton's bounds for $r\ge 7$, and for such $r$ provides the best upper bound, of which we are aware, prior to this paper.
	
	While not stricly on $\RD(n)$, McMullen's work on iterative algorithms \cite{Mc} and his iterative solution of the quintic with Doyle \cite{DM} represent one of the major outgrowths of Arnold's efforts to obstruct solutions of polynomials. More recently, Buhler-Reichstein's formalization of the Kronecker-Klein resolvent problem \cite{BR1,BR2}, and the broader theory of essential dimension that this given rise to, provides the closest contemporary body of work (see e.g. \cite{ReICM}, \cite{Me}, \cite{FKW1}).
	
	The interested reader can find other discussions of the history of the problem in Sylvester and Hammond \cite{SH1}, in Klein \cite{KleinHist}, or more recently in the surveys by Dixmier \cite{Di} and Vitushkin \cite{Vi}. For a contemporary treatment of resolvent degree and its relation to classical problems see also \cite{FW,FKW2}.

\bigskip{\noindent
Dept. of Mathematics, University of California, Irvine\\
E-mail: wolfson@uci.edu

\end{document}